\documentclass[a4paper,11pt]{article}

\usepackage{mathrsfs}
\usepackage{amsmath,amscd}

\usepackage{graphicx}
\usepackage{color}
\usepackage{CJK}
\usepackage{amsmath,amsthm,amssymb,amsfonts}
\usepackage{appendix}

\usepackage[colorlinks,breaklinks,unicode]{hyperref}

\newcommand{\m}{{\bf m}}
\newcommand{\R}{\mathbb R}

\newtheorem{theorem}{Theorem}[section]

\newtheorem{lemma}{Lemma}[section]
\newtheorem{remark}{Remark}[section]
\newtheorem{corollary}{Corollary}[section]
\numberwithin{equation}{section}

\allowdisplaybreaks

\author{Jos\'{e} A. Carrillo\textsuperscript{a}\thanks{E-mail: \href{mailto:carrillo@maths.ox.ac.uk}{\tt carrillo@maths.ox.ac.uk}},
\quad Bin Li\textsuperscript{b}\thanks{ E-mail:  \href{mailto:blimath@163.com}{\tt blimath@163.com}},
\quad Li Xie\textsuperscript{c}\thanks{ E-mail:  \href{mailto:mathxieli@cqnu.edu.cn}{\tt mathxieli@cqnu.edu.cn}}\\
\emph{\small \textsuperscript{a}Mathematical Institute, University of Oxford, Woodstock Road, Oxford, OX2 6GG, United
Kingdom;}\\
\emph{\small \textsuperscript{b}School of Statistics and Data Science, Ningbo University of Technology, Ningbo 315211, China;}\\
\emph{\small \textsuperscript{c}School of Mathematical Sciences, Chongqing Normal University, Chongqing 401331, PR China.}\\
 }

\title{Boundedness and stability of a 2-D parabolic-elliptic system arising in  biological transport networks}

\date{}

\begin{document}
\maketitle
\begin{minipage}{6.2in}\vskip 1mm
 {\small \textbf{Abstract:}
This paper is concerned with the Dirichlet initial-boundary value problem of a 2-D parabolic-elliptic system of the following form
 \begin{equation*}
\left\{
\begin{split}
&\m_t-\kappa\Delta\m+ |\m|^{2(\gamma-1)}\m=(\m\cdot\nabla p)\nabla p,\\
& -\nabla\cdot\left[(\mathbf{I}+\m\otimes\m)\nabla p\right]=S,
\end{split}
\right.
\end{equation*}
 which was proposed to model the formation of biological transport networks. Even if global weak solutions for this system are known to exist, how to improve the regularity of weak solutions is  a challenging problem due to the peculiar cubic nonlinearity  and the  possible elliptic singularity of the system. Global-in-time existence of classical solutions has recently been established showing that finite time singularities cannot emerge in this problem. However, whether or not singularities in infinite time can be precluded was still pending. In this work, we show that classical solutions of the initial-boundary value problem are uniformly bounded in time as long as $\gamma\geq1$ and $\kappa$ is suitably large, closing this gap in the literature. Moreover, uniqueness of classical solutions is also achieved based on the uniform-in-time bounds. Furthermore, it is shown that the corresponding stationary problem possesses a unique classical stationary solution which is semi-trivial, and that is globally exponentially stable, that is, all solutions of the time dependent problem converge exponentially fast to the semi-trivial steady state for $\kappa$ large enough.
}

\vspace{1.5mm}


\vspace{1.5mm}

\end{minipage}

\section{Introduction and main results}

Biological Transport Networks (BTN), such as leaf venation in plants, mammalian circulatory systems, and
neural networks, are crucial to the growth of living systems and therefore their potential mechanisms have drawn great scientific interest; see, for instance, \cite{Morel2010,Malinowski2013,Sedmera2011,Eichmann2005}.
 Mathematical modelling of BTN is traditionally explored in the optimization framework under the assumption that the total material cost is constant \cite{Corson2010,Katifori2010}.
However, living organisms usually adjust their transport network structures to adapt various internal and external stimuli.
Motivated by this,  a purely local dynamic adaptation model was proposed in \cite{Hu2013phys}. This kind of model consists of a system of Ordinary Differential Equations (ODE)  on a graph, and has been
further elaborated or generalized---see, for instance \cite{Kirkegaard2020,King2017,Akita2017,Grawer2015}.  Note that such discrete ODE models require the pre-given network or cellular structures, which cannot reveal the well-known biological principle that transportation drives network growth. To unveil the latter,
the  methods of Partial Differential Equations (PDE) were highlighted in some recent works---see, for instance \cite{Hu2019,Hu2022,Haskovec2015,Albi2016,Albi2017} and also \cite{Haskovec2019CMS,Haskovec2019CPDE,Haskovec2023DCDS,Haskovec2024Me,Haskovec2024Se} for closely related models.

Based on the discrete ODE model \cite{Hu2013phys}, the PDE-based continuum model was introduced to describe the adaptation and evolution of BTN; see, for instance, \cite{Hu2019,Haskovec2015}.
 This kind of model describes  the evolution of the network conductance vector $\m(x,t)$  using a parabolic reaction-diffusion equations, and describes the pressure field $\nabla p$ using a Darcy's type equation with a permeability tensor $\mathbf{I}+\m\otimes\m$,
 where $p(x,t)$  denotes the fluid pressure and $\mathbf{I}$ is the unit matrix.
  Precisely,  such  continuous BTN  model  reads as
\begin{equation}\label{eq-MP0}
\left\{
\begin{split}
&\m_t- \kappa\Delta\m+|\m|^{2(\gamma-1)}\m= (\m\cdot\nabla p)\nabla p,\\
&-\nabla\cdot\left[(\mathbf{I}+\m\otimes\m)\nabla p\right]=S,
\end{split}
\right.
\end{equation}
on a bounded domain $\Omega\subset\mathbb{R}^n$ $(n\ge 1)$, where   $S=S(x)$ denotes the
time-independent sources/sinks, $\kappa>0$ stands for the diffusion coefficient and
 $\gamma$ is a  metabolic
exponent  determined by the particular physical
applications of the type of networks formed, see for instance \cite{Hu2013phys,Haskovec2015,Haskovec2016,Albi2017}.

It is worth pointing out that besides the cubic nonlinear term $(\m\cdot\nabla p)\nabla p$ in the first equation, there is another  peculiar  coupling term $-\nabla\cdot\left[(\m\otimes\m\big)\nabla p\right]$ in the second equation  which may lead to the singularity of the elliptic
operator due to the absence of  {\it  a-priori} bound of $\m$. They both may cause great difficulties in the mathematical analysis of system  \eqref{eq-MP0}. The first works dealing with the analysis of \eqref{eq-MP0} establish the existence of weak solutions to the initial-boundary value problem for \eqref{eq-MP0}.
Precisely, when the homogeneous Dirichlet boundary conditions are imposed, namely $\m|_{\partial\Omega}={\bf0}$ and $p|_{\partial\Omega}=0$, Haskovec et al  proved that such initial-boundary value problem  admits a global weak solution for the case of
$\gamma\geq1$  in \cite{Haskovec2015}, and later  extended  to the case of  $\gamma\geq\frac12$
in \cite{Haskovec2016}, in two- and three-dimensional settings. Therein the weak solutions fulfill that for any $T>0$
\begin{equation}\label{eq-weaks}
\left\{
\begin{split}
&\m\in L^\infty(0,T;(L^{2\gamma}(\Omega)\cap H^1(\Omega))^n),\,\,\,\partial_t\m\in L^2(0,T; (L^2(\Omega))^n),\\
& p\in L^\infty(0,T; H^1(\Omega)),\,\,\, \m\cdot\nabla p\in L^\infty(0,T;L^2(\Omega)).
\end{split}
\right.
\end{equation}
Based on the  existence of weak solutions,
 basic analytical features like the gradient flow structure of \eqref{eq-MP0} and the impact of the model parameters on the geometry and topology
of network formation were investigated in \cite{Albi2016}.

 Note that the cubic nonlinear term $(\m\cdot\nabla p)\nabla p$ in \eqref{eq-MP0} only belongs to $L^\infty(0,T; (L^1(\Omega))^n)$ based upon \eqref{eq-weaks}. Thanks to this, the regularity $\partial_t\m\in L^2(0,T; (L^2(\Omega))^n)$ does not imply that $\Delta\m\in L^2(0,T; (L^2(\Omega))^n)$ directly, and thereby \eqref{eq-weaks} is not enough to trigger a bootstrap argument to improve straightforwardly
 the regularity  of weak solution.
Accordingly,
there is a mathematical and interesting challenge of improving the regularity of weak solutions,
which has also attracted
considerable attention. When $\gamma>\frac12$, Liu and Xu \cite{Liu2018} showed that the
parabolic Hausdorff dimension of the set of singular points cannot
exceed the  space dimension $n$ with $n\leq3$.
 The local
existence and life-span of classical solution in any space dimension was  investigated by Xu in \cite{Xu2017}, but the global existence of classical solution remained unclear.
 The two-dimensional weak solution was proved to be classical on $\Omega\times(0,T)$ for any $T>0$  when the initial data are suitably small in \cite{Xu2019}, and  a Serrin-type blow-up criterion
for three-dimensional classical solution was obtained  in \cite{Li2020}.
Recently,  it was proved that the weak solution is H\"{o}lder continuous in the two-dimensional case via an inequality associated with the Stummel-Kato class of functions, see \cite{Xu2023JDE}, which no longer holds for $n=3$.

We would like to remark that the result in \cite{Xu2023JDE} asserts that weak solutions of  the initial-boundary value problem of  \eqref{eq-MP0} are H\"{o}lder continuous in the two-dimensional setting, which ensures that the solution component  $\m$ belongs to $L^\infty(\Omega\times(0,T))$ for any $T>0$, see also \cite[Eq. (9)]{Xu2019}. This,
 together with \cite[Theorem 1.6, Lemma 2.7]{Xu2017}, entails that
 the initial-boundary value problem of  \eqref{eq-MP0} admits a global classical solution in the two-dimensional setting if $\gamma\geq1$. However, the boundedness of global classical solutions was still unknown.  As for the stationary problem of \eqref{eq-MP0}, i.e.,
\begin{equation}\label{eq-ss}
\left\{
\begin{split}
&- \kappa\Delta\m_\infty+|\m_\infty|^{2(\gamma-1)}\m_\infty= (\m_\infty\cdot\nabla p_\infty)\nabla p_\infty,&x\in\Omega,\\
&-\nabla\cdot\left[(\mathbf{I}+\m_\infty\otimes\m_\infty)\nabla p_\infty\right]=S,&x\in\Omega,\\
&\m_\infty={\bf0},\,\,\,p_\infty=0,&x\in\partial\Omega,
\end{split}
\right.
\end{equation}
 the existence of  weak solutions was established in  \cite{Haskovec2015} for  the $n-$dimensional settings when $n\le3$; particularly, it was shown therein in the one-dimensional setting, the weak solution indeed must be the unique semi-trivial solution whenever the diffusivity $\kappa$ is large enough. 
In addition, in the two-dimensional setting, Xu  \cite{Xu2018}  showed that the  weak stationary solution is also classical  by employing the weakly
monotone function theory and Hardy space methods. However, one does  foresee  some major difficulty in extending it to the three-dimensional setting, just like the time-dependent problem mentioned above. On the other hand, the uniqueness  and stability of the stationary solution were still open in the two-dimensional setting.

Our main result shows the uniform-in-time boundedness and uniqueness of global classical solutions to the following  initial-boundary value problem
\begin{equation}\label{eq-ib0}
\left\{
\begin{split}
&\m_t- \kappa\Delta\m+|\m|^{2(\gamma-1)}\m= (\m\cdot\nabla p)\nabla p,& x \in \Omega,\\
&-\nabla\cdot\left[(\mathbf{I}+\m\otimes\m)\nabla p\right]=S,& x \in \Omega,\\
  &\m(x, 0)=\m_{0}(x), & x \in \Omega,\\
&\m={\bf0},\,\,\, p=0, & x \in \partial \Omega, \,t>0
\end{split}
\right.
\end{equation}
in the two-dimensional setting, as well as the uniqueness and exponential stability of the stationary solution, under the additional assumption that $\kappa$ is suitably large. More precisely, our main result reads as:

\begin{theorem}\label{th-global}
Let $\Omega$ be a bounded domain in $\R^2$  with  smooth boundary,  $\kappa>0$, $\gamma\ge1$, $S\in H^2$ and $\m_0\in\left(\mathcal{C}^{2+\alpha}(\overline{\Omega})\right)^2$ with  $\alpha\in(0,1)$, and let $(\m, p)$ be a classical solution $(\m, p)$ of the initial-boundary value problem \eqref{eq-ib0}. Then there exists $\widetilde{\kappa}>0$ such that whenever $\kappa>\widetilde{\kappa}$,  the following statements hold:
\begin{itemize}
  \item[i)] Any classical solution $(\m, p)$  is uniformly bounded in the sense  that there exists  $C>0$ independent of $t$ such that
\begin{equation}\label{eq-bedd}
\|\m(\cdot,t)\|_{H^2}+ \|(\m\cdot\nabla p)(\cdot,t)\|_{H^2}+\|p(\cdot,t)\|_{H^3}\leq C,\quad t>0.
\end{equation}
 \item[ii)] Any uniformly bounded classical solution is unique.
\item[iii)] The stationary problem \eqref{eq-ss}
        admits a unique classical solution given by $({\bf0},p_\infty^\ast)$ where
        $p_\infty^\ast$ is the unique classical solution of the Dirichlet problem
\begin{equation}\label{eq-ss0}
\left\{
\begin{split}
&-\Delta p^\ast_\infty=S,& x\in\Omega,\\
&p^\ast_\infty=0,& x\in\partial\Omega.
\end{split}
\right.
\end{equation}
   \item[iv)] There exist $C>0$ independent of $t$, and $\mu>0$ independent of $\kappa$ and $t$, such that
\begin{equation}\label{eq-ltime}
\|\m(\cdot,t)\|_{L^\infty}+\|(p-p_\infty^\ast)(\cdot,t)\|_{H^2}\leq Ce^{-\mu\kappa t},\quad t>0.
\end{equation}
 \end{itemize}
\end{theorem}

\begin{remark}\label{r1}
 As we have mentioned above, under the assumptions of Theorem \ref{th-global}, the global existence of classical solutions of  the initial-boundary value problem \eqref{eq-ib0} has been established in \cite{Xu2017, Xu2023JDE}. The novel contribution of Theorem \ref{th-global} is the uniform-in-time boundedness and uniqueness of classical solutions provided that the diffusion coefficient $\kappa$ is  large enough. 
\end{remark}

\begin{remark}
     Theorem \ref{th-global} proves the uniqueness of stationary solution whenever the diffusion $\kappa$ is large enough in the two-dimensional setting, which extends the result in the one-dimensional setting  obtained in \cite{Haskovec2015}. Putting together our main result with the linear stability results in \cite[Section 5.3]{Haskovec2015}, one can rigorously deduce the existence of a bifurcation branch driven by the diffusion coefficient $\kappa$ in two dimensions, generalizing the results in one dimension in \cite[Section 5.1,5.2]{Haskovec2015}.
\end{remark}

\begin{remark}\label{r2}
 Theorem \ref{th-global} also indicates the stability of the stationary solution when  the diffusion coefficient $\kappa$ is  suitably large; the numerical results as in \cite{Albi2016,Hu2019} showed that the stationary solution is unstable when the diffusion coefficient $\kappa$ is  small  enough. Accordingly, it is natural to conjecture that there is a critical value for the diffusion coefficient, possibly the value at which the system becomes linearly unstable \cite[Section 5.3]{Haskovec2015}, but the rigorous proof of the existence of this phase transition and its type (continuous or discontinuous) is still open. 
\end{remark}

 \noindent{\bf Main difficulties and outline of our approach.-}
 A major difficulty of the mathematical analysis of system \eqref{eq-MP0} is that it
features both a cubic nonlinearity and a possible singularity. Specifically, the nonlinear term $(\m\cdot\nabla p)\nabla p$ in \eqref{eq-MP0},  which only belongs to $L^\infty(0,T; (L^1(\Omega))^n)$ based upon \eqref{eq-weaks},  is
still not well-studied in the literature, and consequently both the classical parabolic theories (e.g., \cite{Lady1968}) and the method based on the Stummel-Kato class of functions (e.g., \cite{Xu2023JDE}) no longer apply for getting the  boundedness of classical solutions; the  term $-\nabla\cdot\left[(\m\otimes\m\big)\nabla p\right]$  may lead to the singularity of the elliptic operator due to the absence of an uniform-in-time bound of $\m$ {\it  a-priori}, and this singularity is not included in  \cite{Lieberman1996} and the classical parabolic theories (e.g., \cite{Lady1968ellip}) are of no help for obtaining the  boundedness of classical solutions.

Here, inspired by \cite{Haskovec2015} we begin with the uniform-in-time bound of
\[
\int_\Omega\left(\kappa|\nabla\m|^2+ |\m|^{2\gamma}+|\nabla p|^2+(\m\cdot\nabla p)^2\right)(\cdot,t)dx,
\]
see Lemma \ref{le-PL6t}. After this, one of the key steps of deriving the uniform-in-time bound \eqref{eq-bedd} in Theorem \ref{th-global} is to establish the $L^2$-bound for $\Delta\m$ when $\kappa$ is large enough, see Lemma \ref{le-H2P}. We first prove a crucial control of $\|\Delta p\|_{L^2}$ and $\|\nabla\Delta p\|_{L^2}$ by using $\|\Delta\m\|_{L^2}$ and $\|\Delta\m\|_{L^2}^2$ respectively, with coefficients explicitly dependent on $\kappa$, see Lemma \ref{le-ph2} and Lemma \ref{le-ph3}. This is our key novel estimate in order to bootstrap the uniform-in-time bounds for higher regularity quantities. With the  $L^2$-bound for $\Delta\m$ at hand,  by utilizing elliptic estimates, see Lemma \ref{le-elliptic}, it is easy to get the  desired estimate \eqref{eq-bedd}, see Lemma \ref{le-ubd}.  As for the uniqueness of stationary solutions, we further develop the strategy established in Lemma \ref{le-ph2} and Lemma \ref{le-ph3} in order to obtain a bound on
\[
\kappa\|\Delta\m_\infty\|_{L^2}^2+\|\Delta p_\infty\|_{L^2}^2+\|\nabla(\m_\infty\cdot\nabla p_\infty)\|_{L^2}^2,
\]
for which we can track the explicit dependence on $\kappa$, see Lemma \ref{le-sPH2}. Collecting these results, we can conclude the uniqueness of stationary solution as desired. Finally, to conclude the desired exponential stability, we directly show that $y'(t)+Cy(t)\leq 0$ for $t>0$, where $C>0$ is independent of $t$, and $y(t):=\|\m(\cdot,t)\|_{L^2}^2$. Here, we make extensive use of the uniform-in-time higher regularity estimates obtained in Section 2.

The rest of this paper is structured as follows.  The  uniform-in-time boundedness and uniqueness of classical solutions for the initial-boundary value problem \eqref{eq-ib0} are presented in Section \ref{se-global}. Section \ref{se-stss} is devoted to proving that the stationary problem admits a unique classical solution $({\bf0},p_\infty^\ast)$. The global asymptotic exponential stability of such  steady state is obtained in Section \ref{se-lim}.

\vskip 2mm
 \noindent{\bf Notation.}
 The functional spaces, such as $H^k(\Omega)$ and  $L^q(\Omega)$ endowed with the norms $\| \cdot \|_{H^k(\Omega)}$ and $\| \cdot \|_{L^q(\Omega)}$ respectively, will typically be denoted as $\| \cdot \|_{H^k}$ and $\| \cdot \|_{L^q}$
for convenience. If no confusion arises,  the same notation is used for spaces
of vector-valued functions and of matrix-valued functions, for instance,
\[
\|\nabla\m\|_{L^q}=\sum_{i,j=1}^2\|\partial_j m_i\|_{L^q}\quad \mathrm{with}\quad \partial_j=\partial_{x_j}.
\]

\section{Boundedness and uniqueness of classical solutions}\label{se-global}

The primary goal of this section is to establish the uniform-in-time boundedness of classical solutions to initial-boundary value problem \eqref{eq-ib0}. Namely, we devote to establishing the uniform-in-time bound of $\|\m(\cdot, t)\|_{H^2}+\|(\m\cdot\nabla p)(\cdot,t)\|_{H^2}+\|p(\cdot, t)\|_{H^3}$ stated in \eqref{eq-bedd}, i.e., the first statement of the Theorem \ref{th-global}. On the basis of it, we finally verify the uniqueness of these classical solutions, i.e., the second statement of the Theorem \ref{th-global}.
As a starting point, we begin with a bound on the quantity $\|\nabla p(\cdot,t)\|_{L^2}+\|\m\cdot\nabla p(\cdot,t)\|_{L^2}$.
\begin{lemma}
Let $(\m,p)$ be a classical solution of the initial-boundary value problem \eqref{eq-ib0}. Then there exists a positive constant $C$, independent of $t$ and $\kappa$, such that    
\begin{align}\label{eq-HL6t}
\|\nabla p(\cdot,t)\|_{L^2}+\|\m\cdot\nabla p(\cdot,t)\|_{L^2}\leq C,\quad t>0.
\end{align}
\end{lemma}
\begin{proof}
Testing the second equation in \eqref{eq-ib0} by $p$, and invoking that $(\m\otimes\m)\nabla p=(\m\cdot\nabla p)\m$, we have
\begin{align}\label{eq-1HL6t}
\int_\Omega|\nabla p|^2dx+\int_\Omega(\m\cdot\nabla p)^2dx=\int_\Omega pSdx.
\end{align}
Since $p|_{\partial\Omega}=0$, one can infer from H\"{o}lder's inequality, Poincar\'{e}'s inequality  and Young's inequality that
 \[
 \int_\Omega pSdx\leq \|S\|_{L^2}\|p\|_{L^2}\leq C_1\|S\|_{L^2}\|\nabla p\|_{L^2}\leq \frac12\|\nabla p\|_{L^2}^2+\frac{C_1^2}2\|S\|_{L^2}^2
 \]
with some  $C_1>0$ depending only on $\Omega$.
Substituting it into \eqref{eq-1HL6t}, we directly have the desired estimate \eqref{eq-HL6t}.
\end{proof}

To obtain the bound of $\|\nabla\m\|_{L^2}$, 
inspired by \cite{Haskovec2015}, we define the energy associated to \eqref{eq-ib0} as
\begin{align}\label{energy}
\mathcal{E}[\m,p]:=\frac \kappa2 \int_\Omega|\nabla\m|^2dx+\frac{1}{2\gamma} \int_\Omega|\m|^{2\gamma}dx+\frac12 \int_\Omega|\nabla p|^2dx+\frac12 \int_\Omega(\m\cdot\nabla p)^2dx.
\end{align}

\begin{lemma}\label{le-PL6t}
Let $(\m,p)$ be a classical solution of the initial-boundary value problem \eqref{eq-ib0}. Then  for any $t\geq 0$
\begin{equation}\label{energydecay}
    \frac{d}{dt} \mathcal{E}[\m(\cdot,t),p(\cdot,t)]=-\|\m_t\|_{L^2}^2.
\end{equation}
Moreover, there exists a positive constant $C$, independent of $t$ and $\kappa$, such that
\begin{align}\label{eq-HL6t0-1}
\int_0^t\|\m_s(\cdot,s)\|_{L^2}^2ds+\int_\Omega\left(\kappa|\nabla\m|^2+ |\m|^{2\gamma}\right)(\cdot,t)dx\leq C(1+\kappa),
\end{align} 
in particular,
\begin{align}\label{eq-HL6t0}
\|\nabla\m\|_{L^2}\le C(1+\kappa^{-\frac{1}{2}}),\quad t>0.
\end{align}

\end{lemma}

\begin{proof}
Testing the first equation in \eqref{eq-ib0}  by $\m_t$ leads to
\begin{align}\label{eq-2HL6t}
\|\m_t\|_{L^2}^2+\frac \kappa2\frac{d}{dt}\int_\Omega|\nabla\m|^2dx+\frac{1}{2\gamma}\frac{d}{dt}\int_\Omega|\m|^{2\gamma}dx=\int_\Omega(\m\cdot\nabla p)(\m_t\cdot\nabla p)dx.
\end{align}
To control the term on the right hand, we multiply the second equation in \eqref{eq-ib0} by $p_t$ to get
\begin{align*}
\frac12\frac{d}{dt}\int_\Omega|\nabla p|^2dx+\int_\Omega(\m\cdot\nabla p)(\m\cdot\nabla p_t)dx=\int_\Omega Sp_tdx.
\end{align*}
Since $\m\cdot\nabla p_t=(\m\cdot\nabla p)_t-\m_t\cdot\nabla p$, we have
\begin{align}\label{eq-3HL6t}
\frac12\frac{d}{dt}\int_\Omega|\nabla p|^2dx+\frac12\frac{d}{dt}\int_\Omega(\m\cdot\nabla p)^2dx=\int_\Omega(\m\cdot\nabla p)(\m_t\cdot\nabla p)dx+\int_\Omega Sp_tdx.
\end{align}
Thanks to $S=S(x)$ independent of $t$, invoking \eqref{eq-1HL6t} it holds that
\begin{align*}
\int_\Omega Sp_tdx=\frac{d}{dt}\int_\Omega Spdx=\frac{d}{dt}\int_\Omega|\nabla p|^2dx+\frac{d}{dt}\int_\Omega(\m\cdot\nabla p)^2dx.
\end{align*}
Substituting it into \eqref{eq-3HL6t}, we arrive at
\begin{align*}
\frac12\frac{d}{dt}\int_\Omega|\nabla p|^2dx+\frac12\frac{d}{dt}\int_\Omega(\m\cdot\nabla p)^2dx=-\int_\Omega(\m\cdot\nabla p)(\m_t\cdot\nabla p)dx.
\end{align*}
Inserting it into \eqref{eq-2HL6t}, we obtain \eqref{energydecay} as
\begin{align*}
\|\m_t\|_{L^2}^2+\frac \kappa2\frac{d}{dt}\int_\Omega|\nabla\m|^2dx+\frac{1}{2\gamma}\frac{d}{dt}\int_\Omega|\m|^{2\gamma}dx+\frac12\frac{d}{dt}\int_\Omega|\nabla p|^2dx+\frac12\frac{d}{dt}\int_\Omega(\m\cdot\nabla p)^2dx=0,
\end{align*}
by using the energy definition in \eqref{energy}. We next integrate it in time to deduce that
\begin{align*}
&\int_0^t\|\m_s\|_{L^2}^2ds+\frac \kappa2 \int_\Omega|\nabla\m|^2dx+\frac{1}{2\gamma}\int_\Omega|\m|^{2\gamma}dx+\frac12\int_\Omega|\nabla p|^2dx+\frac12\int_\Omega(\m\cdot\nabla p)^2dx\\
=&\frac \kappa2 \int_\Omega|\nabla\m_0|^2dx+\frac{1}{2\gamma}\int_\Omega|\m_0|^{2\gamma}dx+\frac12\int_\Omega|\nabla p_0|^2dx+\frac12\int_\Omega(\m_0\cdot\nabla p_0)^2dx,
\end{align*}
where $p_0$ satisfies the elliptic problem
\begin{equation*}
\left\{
\begin{split}
&-\nabla\cdot\left[(\mathbf{I}+\m_0\otimes\m_0)\nabla p_0\right]=S,&x\in\Omega,\\
&p_0=0,&x\in\partial\Omega.
\end{split}
\right.
\end{equation*}
It is clear that, for given $S\in H^2$ and $\m_0\in\left(\mathcal{C}^{2+\alpha}(\overline{\Omega})\right)^2$ with $\alpha\in(0,1)$, the above elliptic problem admits a unique classical solution $p_0\in \mathcal{C}^{2}(\overline{\Omega})$ (\cite{Trudinger1983}). Similar to
\eqref{eq-HL6t}, there exists $C_2>0$, independent of $\kappa$, such that
\begin{align*}
\|\nabla p_0\|_{L^2}+\|\m_0\cdot\nabla p_0\|_{L^2}\leq C_2,
\end{align*}
which leads to
\begin{align*}
&\int_0^t\|\m_s\|_{L^2}^2ds+\frac \kappa2 \int_\Omega|\nabla\m|^2dx+\frac{1}{2\gamma}\int_\Omega|\m|^{2\gamma}dx+\frac12\int_\Omega|\nabla p|^2dx+\frac12\int_\Omega(\m\cdot\nabla p)^2dx\\
&\leq\frac \kappa2 \int_\Omega|\nabla\m_0|^2dx+\frac{1}{2\gamma}\int_\Omega|\m_0|^{2\gamma}dx+C_3
\end{align*}
with $C_3>0$ independent of $\kappa$ and $t$. This obviously implies  \eqref{eq-HL6t0-1} and \eqref{eq-HL6t0}.
\end{proof}

To obtain higher-order estimates for $(\m,p)$, we shall frequently utilize elliptic estimates. These classical estimates are summarized in Lemma \ref{le-ellipticesti} in the Appendix. Since $\m|_{\partial\Omega}=0$ and $p|_{\partial\Omega}=0$, application of Lemma \ref{le-ellipticesti} directly yields the following elliptic estimates which will be frequently used in our later analysis.

\begin{lemma}\label{le-elliptic}
There exists $C>0$, independent of $\kappa$ and $t$, such that
\begin{align}\label{eq-elliptic1}
\|p\|_{H^2}\le  C\|\Delta p\|_{L^2},
\end{align}
\begin{align}\label{eq-elliptic2}
\|\m\|_{H^2}\le  C\|\Delta \m \|_{L^2},
\end{align}
\begin{align}\label{eq-esm}
 \|\m\cdot\nabla p\|_{H^2}\leq C\|\Delta(\m\cdot\nabla p)\|_{L^2},
\end{align}
\begin{align}\label{eq-esmH3}
\|\m\|_{H^3}\leq C\left(\|\nabla\Delta \m\|_{L^2}+\|\Delta \m\|_{L^2}\right),
\end{align}
and
\begin{align}\label{eq-esp}
\| p\|_{H^3}\leq C\left(\|\nabla\Delta p\|_{L^2}+\|\Delta p\|_{L^2}\right).
\end{align}
\end{lemma}

Based on these elliptic estimates and the uniform-in-time bound of $\|\nabla p(\cdot, t)\|_{L^2}$, $\|\m\cdot\nabla p\|_{L^2}$ and $\|\nabla \m\|_{L^2}$, we next establish estimates on $\|\Delta p(\cdot, t)\|_{L^2}$ and $\|\nabla\Delta p(\cdot, t)\|_{L^2}$, respectively. They both depend on $\|\Delta \m(\cdot, t)\|_{L^2}$ and $\kappa$, and will play a significant role in establishing the uniform-in-time bound of $\|\m(\cdot, t)\|_{H^2}$, based on the later, the uniform-in-time bound of $\|p(\cdot, t)\|_{H^3}$ can be achieved in return. 

\begin{lemma}\label{le-ph2}
There exists $C>0$, independent of $\kappa$ and $t$, such that
\begin{align}\label{eq-ph2}
\|\Delta p(\cdot,t)\|_{L^2}+\|\nabla(\m\cdot\nabla p)(\cdot,t)\|_{L^2} \leq C\left(1+(1+\kappa^{-\frac12})\|\Delta\m(\cdot,t)\|_{L^2}\right),\quad t>0.
\end{align}
\end{lemma}

\begin{proof}
Testing the second  equation in \eqref{eq-ib0} by $-\Delta p$, we obtain
\begin{align*}
\int_\Omega|\Delta p|^2 dx-\int_\Omega(\m\cdot\nabla p)\m\cdot\nabla\Delta pdx=-\int_\Omega S\Delta pdx,
\end{align*}
where we used the fact that $\m|_{\partial\Omega}={\bf0}$. Integrating by parts again, we arrive at
\begin{align*}
\int_\Omega|\Delta p|^2 dx+\sum_{i=1}^2\int_\Omega \partial_i(\m\cdot\nabla p)\nabla\partial_i p\cdot\m dx
=&-\sum_{i=1}^2\int_\Omega(\m\cdot\nabla p)\nabla\partial_i p\cdot\partial_i\m-\int_\Omega S\Delta pdx.
\end{align*}
Since
\[
\nabla\partial_i p\cdot\m=\partial_i(\m\cdot\nabla p)-\partial_i\m\cdot\nabla p,
\]
it holds that
\begin{align}\label{eq-ph2-t}
&\int_\Omega|\Delta p|^2 dx+\int_\Omega|\nabla(\m\cdot\nabla p)|^2dx\\
=&\sum_{i=1}^2\int_\Omega\partial_i(\m\cdot\nabla p)\nabla p\cdot\partial_i\m dx-\sum_{i=1}^2\int_\Omega(\m\cdot\nabla p)\nabla\partial_i p\cdot\partial_i\m dx-\int_\Omega S\Delta pdx.
\end{align}
which, together with H\"{o}lder's inequality and  the elliptic estimate \eqref{eq-elliptic1} implies
\begin{align*}
&\int_\Omega|\Delta p|^2 dx+\int_\Omega|\nabla(\m\cdot\nabla p)|^2dx\nonumber\\
\leq&\sum_{i=1}^2\left(\|\partial_i(\m\cdot\nabla p)\|_{L^2}\|\nabla p\cdot\partial_i\m\|_{L^2}+\|(\m\cdot\nabla p)\partial_i\m\|_{L^2}\|\nabla\partial_i p\|_{L^2}\right)+ \|S\|_{L^2}\|\Delta p\|_{L^2}\\
\leq &\sum_{i=1}^2\left(\|\partial_i(\m\cdot\nabla p)\|_{L^2}\|\nabla p\cdot\partial_i\m\|_{L^2}+C_1\|(\m\cdot\nabla p)\partial_i\m\|_{L^2}\|\Delta p\|_{L^2}\right)+ \|S\|_{L^2}\|\Delta p\|_{L^2}
\end{align*}
with $C_1>0$ independent of $\kappa$ and $t$. Young's inequality further implies that there exists $C_2>0$ such that
\begin{align}\label{eq-ph2e}
 \int_\Omega|\Delta p|^2 dx+\int_\Omega|\nabla(\m\cdot\nabla p)|^2dx
\leq&C_2\left(\sum_{i=1}^2\left(\|\nabla p\cdot\partial_i\m\|_{L^2}^2+\|(\m\cdot\nabla p)\partial_i\m\|_{L^2}^2\right)+ \|S\|_{L^2}^2\right).
\end{align}
The  Gagliardo-Nirenberg inequality \eqref{GNineq-1} and the elliptic estimates \eqref{eq-elliptic1} and \eqref{eq-elliptic2} entail
\begin{align*}
\|\nabla p\cdot\partial_i\m\|_{L^2}^2\leq& \|\nabla p\|_{L^4}^2\|\partial_i\m\|_{L^4}^2\\
\leq& C_3\|\nabla p\|_{L^2}\|\nabla p\|_{H^1}\|\partial_i\m\|_{L^2}\|\partial_i\m\|_{H^1}\\
\leq &C_4\|\nabla p\|_{L^2}\|\Delta p\|_{L^2}\|\partial_i\m\|_{L^2}\|\Delta\m\|_{L^2}
\end{align*}
with $C_3,\,C_4>0$ independent of $\kappa$ and $t$, which, with the aid of \eqref{eq-HL6t} and  \eqref{eq-HL6t0}, implies
\begin{align*}
\|\nabla p\cdot\partial_i\m\|_{L^2}^2
\leq &C_5(1+\kappa^{-\frac12}) \|\Delta p\|_{L^2} \|\Delta\m\|_{L^2}
\end{align*}
with $C_5>0$ independent of $\kappa$ and $t$. Similarly,   the Poincar\'{e}  inequality, the Gagliardo-Nirenberg inequality \eqref{GNineq-1} and the elliptic estimate \eqref{eq-elliptic2} yield $C_6,\,C_7>0$  such that
\begin{align*}
\|(\m\cdot\nabla p)\partial_i\m\|_{L^2}^2\leq& \|\m\cdot\nabla p\|_{L^4}^2\|\partial_i\m\|_{L^4}^2\\
\leq&C_6\|\m\cdot\nabla p\|_{L^2}\|\m\cdot\nabla p\|_{H^1}\|\partial_i\m\|_{L^2}\|\partial_i\m\|_{H^1}\\
\leq& C_7\|\m\cdot\nabla p\|_{L^2}\|\nabla(\m\cdot\nabla p)\|_{L^2}\|\partial_i\m\|_{L^2}\|\Delta\m\|_{L^2},
\end{align*}
which, combined with \eqref{eq-HL6t} and  \eqref{eq-HL6t0}, leads to
\begin{align*}
\|(\m\cdot\nabla p)\partial_i\m\|_{L^2}^2\leq&C_8(1+\kappa^{-\frac12})\|\nabla(\m\cdot\nabla p)\|_{L^2}\|\Delta\m\|_{L^2}
\end{align*}
with $C_8>0$ independent of $\kappa$ and $t$. This together with \eqref{eq-ph2e} implies that there eixsts $C_9>0$ independent of $\kappa$ and $t$, such that
\begin{align*}
\int_\Omega|\Delta p|^2 dx+\int_\Omega|\nabla(\m\cdot\nabla p)|^2dx
\leq& C_9(1+\kappa^{-\frac12}) \|\Delta\m\|_{L^2}\left(\|\nabla(\m\cdot\nabla p)\|_{L^2}+\|\Delta p\|_{L^2}\right)+C_9,
\end{align*}
which, with the help of Young's inequality, gives us
\begin{align*}
\int_\Omega|\Delta p|^2 dx+\int_\Omega|\nabla(\m\cdot\nabla p)|^2dx
\leq& C_{10}(1+\kappa^{-1}) \|\Delta\m\|_{L^2}^2+C_{10}
\end{align*}
with some $C_{10}>0$ independent of $\kappa$ and $t$. Accordingly,  we have \eqref{eq-ph2} as desired.
\end{proof}

\begin{lemma}\label{le-ph3}
There exists $C>0$, independent of $\kappa$ and $t$, such that
\begin{align}\label{eq-ph3}
\|\nabla\Delta p(\cdot,t)\|_{L^2}+\|\Delta(\m\cdot\nabla p)(\cdot,t)\|_{L^2} \leq C\left(1+(1+\kappa^{-1})\|\Delta\m(\cdot,t)\|_{L^2}^2\right),\quad t>0.
\end{align}
\end{lemma}

\begin{proof}
We first differentiate the second equation in \eqref{eq-ib0} with respect to $x_i$   to get
\[
-\Delta\partial_ip-\nabla\cdot\left[\partial_i(\m\cdot\nabla p)\m+(\m\cdot\nabla p)\partial_i\m\right]=\partial_iS.
\]
  Testing it by $-\Delta\partial_ip$, yields that
\begin{align}\label{eq-ph3-1}
&\int_\Omega|\nabla\Delta p|^2dx-\sum_{i=1}^2\int_\Omega\left[\partial_i(\m\cdot\nabla p)\m+(\m\cdot\nabla p)\partial_i\m\right]\cdot\nabla\Delta\partial_ipdx=-\int_\Omega\nabla S\cdot\nabla\Delta pdx,
\end{align}
due to the fact that $\m|_{\partial\Omega}={\bf0}$.
By direct calculation invoking the integration by parts, we have
\begin{align*}
&-\sum_{i=1}^2\int_\Omega\left[\partial_i(\m\cdot\nabla p)\m+(\m\cdot\nabla p)\partial_i\m\right]\cdot\nabla\Delta\partial_ipdx\\
=&\sum_{i=1}^2\int_\Omega\left[\partial_{ii}^2(\m\cdot\nabla p)\m+2\partial_i(\m\cdot\nabla p)\partial_i\m+(\m\cdot\nabla p)\partial_{ii}^2\m\right]\cdot\nabla\Delta pdx\\
=&\int_\Omega \Delta(\m\cdot\nabla p)\m\cdot\nabla\Delta pdx+\int_\Omega\left[2\sum_{i=1}^2\partial_i(\m\cdot\nabla p)\partial_i\m+(\m\cdot\nabla p)\Delta\m\right]\cdot\nabla\Delta pdx,
\end{align*}
where we also used the fact that $\m|_{\partial\Omega}={\bf0}$. Note that
\[
\m\cdot\nabla\Delta p=\Delta(\m\cdot\nabla p)-\Delta\m\cdot\nabla p-2\sum_{l=1}^2\partial_l\m\cdot\nabla\partial_lp,
\]
and thereby
\begin{equation*}
\begin{split}
&-\sum_{i=1}^2\int_\Omega\left[\partial_i(\m\cdot\nabla p)\m+(\m\cdot\nabla p)\partial_i\m\right]\cdot\nabla\Delta\partial_ipdx\\
=&\int_\Omega\big(\Delta(\m\cdot\nabla p)\big)^2dx-\int_\Omega \Delta(\m\cdot\nabla p)\left(\Delta\m\cdot\nabla p+2\sum_{l=1}^2\partial_l\m\cdot\nabla\partial_lp\right)dx\\
&+\int_\Omega\left[2\sum_{i=1}^2\partial_i(\m\cdot\nabla p)\partial_i\m+(\m\cdot\nabla p)\Delta\m\right]\cdot\nabla\Delta pdx.
\end{split}
\end{equation*}
Substituting it into \eqref{eq-ph3-1}, we arrive at
\begin{align*}
&\int_\Omega|\nabla\Delta p|^2dx+\int_\Omega\big(\Delta(\m\cdot\nabla p)\big)^2dx\\
=&\int_\Omega \Delta(\m\cdot\nabla p)\left(\Delta\m\cdot\nabla p+2\sum_{l=1}^2\partial_l\m\cdot\nabla\partial_lp\right)dx\\
&-\int_\Omega\left[2\sum_{i=1}^2\partial_i(\m\cdot\nabla p)\partial_i\m+(\m\cdot\nabla p)\Delta\m\right]\cdot\nabla\Delta pdx-\int_\Omega\nabla S\cdot\nabla\Delta pdx,
\end{align*}
which, combined with H\"{o}lder's inequality and Young's inequality, ensures
\begin{equation}\label{eq-ph3-2}
\begin{split}
&\int_\Omega|\nabla\Delta p|^2dx+\int_\Omega\big(\Delta(\m\cdot\nabla p)\big)^2dx\\
\leq&C_1\left(1 +\|\Delta\m\cdot\nabla p\|_{L^2}^2+\sum_{i=1}^2\big(\|\partial_i\m\cdot\nabla\partial_ip\|_{L^2}^2+\|\partial_i(\m\cdot\nabla p)\partial_i\m\|_{L^2}^2\big)+\|(\m\cdot\nabla p)\Delta\m\|_{L^2}^2\right)
\end{split}
\end{equation}
with some $C_1>0$ independent of $\kappa$ and $t$. We now estimate the terms on the right hand  one by one. First, on the basis of the H\"{o}lder inequality, the  Gagliardo-Nirenberg inequality \eqref{GNineq-2} and the uniform bound provided by \eqref{eq-HL6t}, there exist $C_2, \,C_3>0$, independent of $\kappa$ and $t$,  fulfilling that
\begin{align*}
\|\Delta\m\cdot\nabla p\|_{L^2}^2\leq\|\Delta\m\|_{L^2}^2\|\nabla p\|_{L^\infty}^2
\leq C_2\|\Delta\m\|_{L^2}^2\|\nabla p\|_{L^2}\|\nabla p\|_{H^2}\leq C_3\|\Delta\m\|_{L^2}^2\|\nabla p\|_{H^2}.
\end{align*}
Similarly, due to \eqref{eq-HL6t} there holds
\begin{align*}
\|(\m\cdot\nabla p)\Delta\m\|_{L^2}^2\leq& \|\m\cdot\nabla p\|_{L^\infty}^2\|\Delta\m\|_{L^2}^2\\
\leq&C_4\|\m\cdot\nabla p\|_{L^2}\|\m\cdot\nabla p\|_{H^2}\|\Delta\m\|_{L^2}^2\\
\leq &C_5\|\m\cdot\nabla p\|_{H^2}\|\Delta\m\|_{L^2}^2.
\end{align*}
Second, it follows from   H\"{o}lder's inequality,  Gagliardo-Nirenberg inequalities \eqref{GNineq-1} and \eqref{GNineq-3},  uniform bounds provided by \eqref{eq-HL6t} and \eqref{eq-HL6t0} that
\begin{align*}
\|\partial_i\m\cdot\nabla\partial_ip\|_{L^2}^2\leq& \|\partial_i\m\|_{L^4}^2\|\nabla\partial_ip\|_{L^4}^2\\
\leq& C_6\|\partial_i\m\|_{L^2}\|\partial_i\m\|_{H^1}\|\partial_ip\|_{L^2}^\frac12\|\partial_ip\|_{H^2}^\frac32\\
\leq&C_7(1+\kappa^{-\frac12})\|\partial_i\m\|_{H^1}\|\partial_ip\|_{H^2}^\frac32
\end{align*}
with some $C_6,\,C_7>0$ independent of $\kappa$ and $t$, and  similarly,
\begin{align*}
\|\partial_i(\m\cdot\nabla p)\partial_i\m\|_{L^2}^2\leq&\|\partial_i(\m\cdot\nabla p)\|_{L^4}^2\|\partial_i\m\|_{L^4}^2\\
\leq&C_8\|\m\cdot\nabla p\|_{L^2}^\frac12\|\partial_i(\m\cdot\nabla p)\|_{H^1}^\frac32\|\partial_i\m\|_{L^2}\|\partial_i\m\|_{H^1}\\
\leq&C_9(1+\kappa^{-\frac12})\|\partial_i(\m\cdot\nabla p)\|_{H^1}^\frac32\|\partial_i\m\|_{H^1}
\end{align*}
with some $C_8,\,C_9>0$ independent of $\kappa$ and $t$. Substituting the above into \eqref{eq-ph3-2}, we have
\begin{align*}
&\int_\Omega|\nabla\Delta p|^2dx+\int_\Omega\big(\Delta(\m\cdot\nabla p)\big)^2dx\\
\leq&C_{10}\left(1 +\|\Delta\m\|_{L^2}^2\left(\|\nabla p\|_{H^2}+\|\m\cdot\nabla p\|_{H^2}\right)+(1+\kappa^{-\frac12})\|\nabla\m\|_{H^1}\left(\|\nabla p\|_{H^2}^\frac32+\|\nabla(\m\cdot\nabla p)\|_{H^1}^\frac32\right)\right),
\end{align*}
where $C_{10}>0$ is independent of $\kappa$ and $t$.
Applications of elliptic estimates \eqref{eq-elliptic2}, \eqref{eq-esm} and \eqref{eq-esp} further entail
  \begin{align*}
&\int_\Omega|\nabla\Delta p|^2dx+\int_\Omega\big(\Delta(\m\cdot\nabla p)\big)^2dx\\
\leq&C_{11}\left[1 +\|\Delta\m\|_{L^2}^2\big(\|\nabla\Delta p\|_{L^2}+\|\Delta p\|_{L^2}+\|\Delta(\m\cdot\nabla p)\|_{L^2}\big)\right]\\
&+C_{11}(1+\kappa^{-\frac12})\|\Delta\m\|_{L^2}\left(\|\nabla\Delta p\|_{L^2}^\frac32+\|\Delta p\|_{L^2}^\frac32+\|\Delta(\m\cdot\nabla p)\|_{L^2}^\frac32\right),
\end{align*}
with some $C_{11}>0$  independent of $\kappa$ and $t$,  which, using Young's inequality, leads to
    \begin{align*}
&\int_\Omega|\nabla\Delta p|^2dx+\int_\Omega\big(\Delta(\m\cdot\nabla p)\big)^2dx\\
\leq&C_{12}\left(1 +\|\Delta\m\|_{L^2}^4+\|\Delta\m\|_{L^2}^2\|\Delta p\|_{L^2}\right)+C_{13}(1+\kappa^{-2})\|\Delta\m\|_{L^2}^4 \\
&+C_{12}(1+\kappa^{-\frac12})\|\Delta\m\|_{L^2}\left(1+\|\Delta p\|_{L^2}^\frac32\right)
\end{align*}
with some $C_{12}>0$ also independent of $\kappa$ and $t$. Based on it, using \eqref{eq-ph2} and Young's inequality, we can find $C_{13},\,C_{14}>0$ independent of $\kappa$ and $t$ such that
   \begin{align*}
&\int_\Omega|\nabla\Delta p|^2dx+\int_\Omega\big(\Delta(\m\cdot\nabla p)\big)^2dx\\
\leq&C_{13}\left(1 +\|\Delta\m\|_{L^2}^4+\|\Delta\m\|_{L^2}^2\left(1+(1+\kappa^{-\frac12})\|\Delta\m\|_{L^2}\right)\right)+C_{13}(1+\kappa^{-2})\|\Delta\m\|_{L^2}^4 \\
&+C_{13}(1+\kappa^{-\frac12})\|\Delta\m\|_{L^2}\left(1+(1+\kappa^{-\frac34})\|\Delta\m\|_{L^2}^\frac{3}{2}\right)\\
\leq &C_{14}\left(1+(1+\kappa^{-2})\|\Delta\m\|_{L^2}^4\right),
\end{align*}
 which entails \eqref{eq-ph3} as desired.
\end{proof}

With Lemma \ref{le-ph2} and Lemma \ref{le-ph3} at hand, we can now follow a reasoning invoking the
energy method to  obtain the  uniform-in-time boundedness for $\|\Delta\m(\cdot,t)\|_{L^2}$.

\begin{lemma}\label{le-H2P}
 There exists  $\hat{\kappa}_0>0$  such that, whenever $\kappa\ge\hat{\kappa}_0$, then there exists $C>0$ (independent of $t$ and $\kappa$) such that if $\gamma=1$, then
\begin{align}\label{eq-H2in}
\|\Delta\m(\cdot,t)\|_{L^2} \leq C(1+\kappa^{-\frac12}),\quad t>0,
\end{align}
and if $\gamma>1$, then
\begin{align}\label{eq-H2ing}
\|\Delta\m(\cdot,t)\|_{L^2} \leq C\left(1+\kappa^{-\frac12-\frac{4\gamma^2-3\gamma+2}{3}-\frac{3}{8(\gamma-1)}}\right),\quad t>0.
\end{align}
\end{lemma}

\begin{proof}
If  $\gamma=1$, then the first equation of \eqref{eq-ib0} can be rewritten as
\begin{align*}
\m_t- \kappa\Delta\m+\m= (\m\cdot\nabla p)\nabla p.
\end{align*}
Based on it, a direct calculation invoking integration by
parts shows that
\begin{equation}\label{eq-H2in-1}
\begin{split}
&\frac12\frac{d}{dt}\int_\Omega|\Delta\m|^2dx+\kappa\int_\Omega|\nabla\Delta\m|^2 dx+\int_\Omega|\Delta\m|^2dx\\
=&\frac\kappa2\int_{\partial\Omega}\frac{\partial(|\Delta\m|^2)}{\partial\nu}d\sigma-\sum_{i=1}^2\int_\Omega\big(\partial_i(\m\cdot\nabla p)\nabla p\cdot\partial_i\Delta\m +(\m\cdot\nabla p)\nabla\partial_i p\cdot\partial_i\Delta\m\big) dx\\
&+\sum_{i=1}^2\int_{\partial\Omega}\partial_i((\m\cdot\nabla p)\nabla p)\nu_i\Delta\m d\sigma,
\end{split}
\end{equation}
where  $\nu_i$ stands for the $i$-th component of $\nu$. To deal with the boundary integrals, we recall  $\m|_{\partial\Omega}={\bf0}$, and thereby get that
\begin{align}\label{eq-bmc}
\Delta\m|_{\partial\Omega}=\kappa^{-1}\left[\m_t+\m-(\m\cdot\nabla p)\nabla p\right]_{\partial\Omega}={\bf0}.
\end{align}
It ensures that
\begin{align*}
\frac\kappa2\int_{\partial\Omega}\frac{\partial(|\Delta\m|^2)}{\partial\nu}d\sigma+\sum_{i=1}^2\int_{\partial\Omega}\partial_i((\m\cdot\nabla p)\nabla p)\nu_i\Delta\m d\sigma=0,
\end{align*}
accordingly,
\begin{align*}
&\frac12\frac{d}{dt}\int_\Omega|\Delta\m|^2dx+\kappa\int_\Omega|\nabla\Delta\m|^2 dx+\int_\Omega|\Delta\m|^2dx\\
=&-\sum_{i=1}^2\int_\Omega\big(\partial_i(\m\cdot\nabla p)\nabla p\cdot\partial_i\Delta\m +(\m\cdot\nabla p)\nabla\partial_i p\cdot\partial_i\Delta\m\big) dx.
\end{align*}
An application of  Young's inequality, H\"{o}lder's inequality and  Gagliardo-Nirenberg's inequalities \eqref{GNineq-1} and \eqref{GNineq-3}  provides $C_1>0$, independent of $\kappa$ and $t$, such that
\begin{align*}
&\frac12\frac{d}{dt}\int_\Omega|\Delta\m|^2dx+\frac\kappa2\int_\Omega|\nabla\Delta\m|^2 dx+\int_\Omega|\Delta\m|^2dx\\
\leq& \frac1\kappa\sum_{i=1}^2\left(\|\partial_i(\m\cdot\nabla p)\nabla p\|_{L^2}^2+\|(\m\cdot\nabla p)\nabla\partial_i p\|_{L^2}^2\right)\\
\leq& \frac1\kappa\left(\|\nabla(\m\cdot\nabla p)\|_{L^4}^2\|\nabla p\|_{L^4}^2+\|\m\cdot\nabla p\|_{L^4}^2\|D^2 p\|_{L^4}^2\right)\\
\leq&\frac{ C_1}\kappa\left(\|\m\cdot\nabla p\|_{L^2}^\frac12\|\m\cdot\nabla p\|_{H^2}^\frac32\|\nabla p\|_{L^2}\|\nabla p\|_{H^1}+\|\m\cdot\nabla p\|_{L^2}\|\m\cdot\nabla p\|_{H^1}\|\nabla p\|_{L^2}^\frac12\|\nabla p\|_{H^2}^\frac32\right),
\end{align*}
which, by means of \eqref{eq-HL6t}, leads to
\begin{align*}
&\frac12\frac{d}{dt}\int_\Omega|\Delta\m|^2dx+\frac\kappa2\int_\Omega|\nabla\Delta\m|^2 dx+\int_\Omega|\Delta\m|^2dx\\
\leq&\frac{C_2}{\kappa}\left(\|\m\cdot\nabla p\|_{H^2}^\frac32\|\nabla p\|_{H^1}+\|\m\cdot\nabla p\|_{H^1}\|\nabla p\|_{H^2}^\frac32\right),
\end{align*}
with some $C_2>0$. Using Poincar\'{e}'s inequality since $\m|_{\partial\Omega}=0$, the elliptic estimates \eqref{eq-elliptic1}, \eqref{eq-esm} and \eqref{eq-esp}, we can find $C_3>0$ such that
\begin{align*}
&\frac12\frac{d}{dt}\int_\Omega|\Delta\m|^2dx+\frac\kappa2\int_\Omega|\nabla\Delta\m|^2 dx+\int_\Omega|\Delta\m|^2dx\\
\leq&\frac{C_3}{\kappa}\left(\|\Delta(\m\cdot\nabla p)\|_{L^2}^\frac32\|\Delta p\|_{L^2}+\|\nabla(\m\cdot\nabla p)\|_{L^2}\left(\|\nabla\Delta p\|_{L^2}^\frac32+\|\Delta p\|_{L^2}^\frac32\right)\right),
\end{align*}
which, together with \eqref{eq-ph2}, \eqref{eq-ph3} and Young's inequality, entails
\begin{align}\label{eq-esmH3-}
\frac12\frac{d}{dt}\int_\Omega|\Delta\m|^2dx+\frac\kappa2\int_\Omega|\nabla\Delta\m|^2 dx+\int_\Omega|\Delta\m|^2dx
\leq&\frac{C_4}{\kappa}\left( 1+(1+\kappa^{-2})\|\Delta\m\|_{L^2}^4\right)
\end{align}
with $C_4>0$ independent of $\kappa$ and $t$.

The Gagliardo-Nirenberg inequality \eqref{GNineq-4} and the elliptic estimate \eqref{eq-esmH3} ensure that there exist positive constants $C_5$ and $C_6$, independent of $\kappa$ and $t$, such that
\begin{align*}
 \|\Delta\m\|_{L^2}^4\leq &C_5\|\nabla\m\|_{L^2}^2\|\nabla\m\|_{H^2}^2\\
                      \le& C_6\|\nabla\m\|_{L^2}^2\left(\|\Delta\m\|_{L^2}^2+\|\nabla\Delta\m\|_{L^2}^2\right)
\end{align*}
Substituting this into \eqref{eq-esmH3-} and using \eqref{eq-HL6t0}, we have
\begin{align}\label{eq-esmH3-0}
\frac{d}{dt}\int_\Omega|\Delta\m|^2dx+\kappa\int_\Omega|\nabla\Delta\m|^2 dx+2\int_\Omega|\Delta\m|^2dx
\leq&\frac{C_7}{\kappa}\left( 1+(1+\kappa^{-3})\left(\|\Delta\m\|_{L^2}^2+\|\nabla\Delta\m\|_{L^2}^2\right)\right)
\end{align}
with $C_7>0$ independent of $\kappa$ and $t$.
Note that one can find $\kappa_1>0$ such that if $\kappa\geq\kappa_1$, then
\[
C_7(1+\kappa^{-3})\leq \kappa,\,\,\,\,C_7(1+\kappa^{-3})\leq \frac12\kappa^2.
\]
This leads to
\begin{align*}
\frac{d}{dt}\int_\Omega|\Delta\m|^2dx+\frac\kappa2\int_\Omega|\nabla\Delta\m|^2 dx+\int_\Omega|\Delta\m|^2dx
\leq&\frac{C_7}{\kappa},
\end{align*}
which by a standard ODE argument ensures that
\[
\|\Delta \m\|_{L^2}^2\le\max\left\{\|\Delta\m_0\|_{L^2}^2,\frac{C_7}{\kappa}\right\},\quad t>0.
\]
This implies \eqref{eq-H2in}  as desired.

Next we handle the case $\gamma>1$. In this case, there holds
\begin{equation}\label{eq-H2in-1-2}
\begin{split}
&\frac12\frac{d}{dt}\int_\Omega|\Delta\m|^2dx+\kappa\int_\Omega|\nabla\Delta\m|^2 dx+\int_\Omega \Delta\m\cdot\Delta(|\m|^{2\gamma-1}\m)\\
=&\frac\kappa2\int_{\partial\Omega}\frac{\partial(|\Delta\m|^2)}{\partial\nu}d\sigma-\sum_{i=1}^2\int_\Omega\big(\partial_i(\m\cdot\nabla p)\nabla p\cdot\partial_i\Delta\m +(\m\cdot\nabla p)\nabla\partial_i p\cdot\partial_i\Delta\m\big) dx\\
&+\sum_{i=1}^2\int_{\partial\Omega}\partial_i((\m\cdot\nabla p)\nabla p)\nu_i\Delta\m d\sigma.
\end{split}
\end{equation}
Indeed, the unique difference between \eqref{eq-H2in-1-2} and \eqref{eq-H2in-1} is that
 the third term $\int_\Omega |\Delta\m|^2$ on the left hand side in  \eqref{eq-H2in-1}  is replaced by $\int_\Omega \Delta\m\cdot\Delta(|\m|^{2\gamma-1}\m)$. Hence, we only need to deal with it. By the integration by parts, it can be denoted by
\begin{align*}
\Pi:=&\sum_{i=1}^2\int_{\partial\Omega}\partial_i(|\m|^{2(\gamma-1)}\m)\partial\nu_i\cdot\Delta\m d\sigma-\sum_{i=1}^2\int_\Omega\partial_i(|\m|^{2(\gamma-1)}\m)\cdot(\partial_i\Delta\m)dx.
\end{align*}
 Similar to \eqref{eq-bmc}, we have $\Delta\m|_{\partial\Omega}={\bf0}$, and thus
\begin{align*}
\Pi=&-\sum_{i=1}^2\int_\Omega\partial_i(|\m|^{2(\gamma-1)}\m)\cdot(\partial_i\Delta\m)dx\\
\leq&(2\gamma-1)\|\m\|_{L^\infty}^{2(\gamma-1)}\|\nabla\m\|_{L^2}\|\nabla\Delta\m\|_{L^2}.
\end{align*}
An application of the Gagliardo-Nirenberg inequality \eqref{GNineq-5}, the Poincar\'{e} inequality and the elliptic estimate \eqref{eq-elliptic2} shows that for any $q>1$ there exists $C_8, C_9>0$  such that
\[
\|\m\|_{L^\infty}^{2(\gamma-1)}\leq C_8\|\m\|_{L^q}^{\frac{2q(\gamma-1)}{q+2}}\|\m\|_{H^2}^\frac{4(\gamma-1)}{q+2} \le C_9\|\nabla\m\|_{L^2}^{\frac{2q(\gamma-1)}{q+2}}\|\Delta \m\|_{L^2}^\frac{4(\gamma-1)}{q+2},
\]
which, with the aid of \eqref{eq-HL6t0}, leads to
 \begin{align*}
\Pi
\leq&C_{10}\left(1+\kappa^{-\frac{2q\gamma-q+2}{2(q+2)}}\right)\|\Delta \m\|_{L^2}^\frac{4(\gamma-1)}{q+2}\|\nabla\Delta\m\|_{L^2}
\end{align*}
with some $C_{10}>0$ independent of $\kappa$ and $t$.
 Based on this, a counterpart of \eqref{eq-esmH3-0} fulfills
\begin{align*}
\frac{d}{dt}\int_\Omega|\Delta\m|^2dx+\kappa\int_\Omega|\nabla\Delta\m|^2 dx
\leq&\frac{C_{11}}{\kappa}\left( 1+(1+\kappa^{-3})\left(\|\Delta\m\|_{L^2}^2+\|\nabla\Delta\m\|_{L^2}^2\right)\right)\\
&+C_{11}\left(1+\kappa^{-\frac{2q\gamma-q+2}{2(q+2)}}\right)\|\Delta \m\|_{L^2}^\frac{4(\gamma-1)}{q+2}\|\nabla\Delta\m\|_{L^2}
\end{align*}
with some $C_{11}>0$ independent of $\kappa$ and $t$. Pick $\kappa_2>0$ such that  $\frac{C_{11}}{\kappa} (1+\kappa^{-3})\leq \frac\kappa4$ whenever $\kappa\geq \kappa_2$, and hence it follows from Young's inequality that
 \begin{align}\label{eq-mh3p}
\frac{d}{dt}\int_\Omega|\Delta\m|^2dx+\frac{\kappa}{2}\int_\Omega|\nabla\Delta\m|^2 dx%\nonumber\\
\leq&\frac{C_{11}}{\kappa}+\frac{\kappa}4\|\Delta\m\|_{L^2}^2 +\frac{C_{12}}\kappa\left(1+\kappa^{-\frac{2q\gamma-q+2}{q+2}}\right)\|\Delta \m\|_{L^2}^\frac{8(\gamma-1)}{q+2}
\end{align}
with some $C_{12}>0$ independent of $\kappa$ and $t$. To control the term $\|\Delta\m\|_{L^2}$ on the right hand, we seek the
evolution of $\|\nabla\m(\cdot,t)\|_{L^2}$.
Indeed, testing the first equation in \eqref{eq-ib0} by $-\Delta\m$ and using  $\m|_{\partial\Omega}=0$, we arrive at
\begin{align*}
&\frac12\frac{d}{dt}\int_\Omega|\nabla\m|^2dx+\kappa\int_\Omega|\Delta\m|^2 dx\\
=&-\int_\Omega|\m|^{2(\gamma-1)}|\nabla\m|^2dx-2(\gamma-1)\sum_{i=1}^2\int_\Omega|\m|^{2(\gamma-2)}(\m\cdot\partial_i\m)^2dx+\int_\Omega (\m\cdot\nabla p)\nabla p \cdot(-\Delta\m)dx,
\end{align*}
which, with the help of H\"{o}lder's inequality,  the Gagliardo-Nirenberg inequality \eqref{GNineq-1} and the fact that $\gamma>1$, leads to
\begin{align*}
\frac{d}{dt}\int_\Omega|\nabla\m|^2dx+2\kappa\int_\Omega|\Delta\m|^2 dx\leq&2\|\m\cdot\nabla p\|_{L^4}\|\nabla p\|_{L^4}\|\Delta\m\|_{L^2}\\
\leq&C_{13}\|\m\cdot\nabla p\|_{L^2}^\frac12\|\m\cdot\nabla p\|_{H^1}^\frac12\|\nabla p\|_{L^2}^\frac12\|\nabla p\|_{H^1}^\frac12\|\Delta\m\|_{L^2}
\end{align*}
with some $C_{13}>0$ independent of $\kappa$ and $t$. Applications of \eqref{eq-HL6t}, Poincar\'{e}'s inequality and the elliptic estimate \eqref{eq-elliptic1} further ensure that there exists $C_{14},\,C_{15}>0$ such that
\begin{align*}
\frac{d}{dt}\int_\Omega|\nabla\m|^2dx+2\kappa\int_\Omega|\Delta\m|^2 dx\leq&C_{14}\|\m\cdot\nabla p\|_{H^1}^\frac12\|\nabla p\|_{H^1}^\frac12\|\Delta\m\|_{L^2}\\
\leq&C_{15}\|\nabla(\m\cdot\nabla p)\|_{L^2}^\frac12\|\Delta p\|_{L^2}^\frac12\|\Delta\m\|_{L^2},
\end{align*}
which, in accordance with \eqref{eq-ph2}, asserts that
\begin{align*}
\frac{d}{dt}\int_\Omega|\nabla\m|^2dx+2\kappa\int_\Omega|\Delta\m|^2 dx
\leq&C_{16}\left(1+(1+\kappa^{-\frac12})\|\Delta\m\|_{L^2}^2\right)
\end{align*}
with some $C_{16}>0$ independent of $\kappa$ and $t$. We can find $\kappa_3\geq\kappa_2$ such that, whenever $\kappa\geq\kappa_3$, it holds that $\kappa\geq C_{16}(1+\kappa^{-\frac12})$, and thereby
\begin{align*}
\frac{d}{dt}\int_\Omega|\nabla\m|^2dx+\kappa\int_\Omega|\Delta\m|^2 dx
\leq&C_{16}.
\end{align*}
 This, together with \eqref{eq-mh3p}, implies that
 \begin{align}\label{eq-mh3p0}
&\frac{d}{dt}\int_\Omega\big(|\Delta\m|^2+|\nabla\m|^2\big)  dx+\frac34\kappa\int_\Omega|\Delta\m|^2dx+\frac{\kappa}{2}\int_\Omega|\nabla\Delta\m|^2 dx\nonumber\\
\leq&C_{17}(1+\kappa^{-1}) +\frac{C_{17}}\kappa\left(1+\kappa^{-\frac{2q\gamma-q+2}{q+2}}\right)\|\Delta \m\|_{L^2}^\frac{8(\gamma-1)}{q+2}
\end{align}
with some $C_{17}>0$  independent of $\kappa$ and $t$. For fixed $\gamma>1$, we can take $q=4\gamma-3$, which ensures that $\frac{8(\gamma-1)}{q+2}=\frac{8(\gamma-1)}{4\gamma-1}<2$. Hence, an application of  Young's inequality shows
\begin{align*}
&C_{17}\kappa^{-1}\left(1+\kappa^{-\frac{2q\gamma-q+2}{q+2}}\right)\|\Delta \m\|_{L^2}^\frac{8(\gamma-1)}{q+2}\\
=&C_{17}\kappa^{-1}\left(1+\kappa^{-\frac{(2\gamma-1)(4\gamma-3)+2}{4\gamma-1}}\right)\|\Delta \m\|_{L^2}^\frac{8(\gamma-1)}{4\gamma-1}\\
\leq& \frac\kappa4\|\Delta \m\|_{L^2}^2+\frac{3}{4\gamma-1}\left(\frac{16(\gamma-1)}{(4\gamma-1)\kappa}\right)^{\frac{3}{4(\gamma-1)}}\left(C_{17}\kappa^{-1}\left(1+\kappa^{-\frac{(2\gamma-1)(4\gamma-3)+2}{4\gamma-1}}\right)\right)^{\frac{4\gamma-1}{3}}\\
\leq&\frac\kappa4\|\Delta \m\|_{L^2}^2+ C_{18}\kappa^{-\frac{4\gamma-1}{3}-\frac{3}{4(\gamma-1)}}\left(1+\kappa^{-\frac{(2\gamma-1)(4\gamma-3)+2}{3}}\right)
\end{align*}
with some $C_{18}>0$ independent of $\kappa$ and $t$. With the help of this, \eqref{eq-mh3p0} reduces to
 \begin{align*}
\frac{d}{dt}\int_\Omega\big(|\Delta\m|^2+|\nabla\m|^2\big) dx+\frac\kappa2\int_\Omega|\Delta\m|^2 dx
\leq& \,C_{17}(1+\kappa^{-1})\\
&+C_{18}\kappa^{-\frac{4\gamma-1}{3}-\frac{3}{4(\gamma-1)}}\left(1+\kappa^{-\frac{(2\gamma-1)(4\gamma-3)+2}{3}}\right).
\end{align*}
Notice that the elliptic estimate \eqref{eq-elliptic2} implies that there exists a positive constant $C_{19}$, independent of $t$ and $\kappa$, such that
$$\frac12\int_\Omega|\Delta\m|^2 dx\ge C_{19}\int_\Omega\big(|\Delta\m|^2+|\nabla\m|^2\big) dx.$$
Thus,
 \begin{align*}
&\frac{d}{dt}\int_\Omega\big(|\Delta\m|^2+|\nabla\m|^2 \big)dx+C_{19}\kappa\int_\Omega\big(|\Delta\m|^2+|\nabla\m|^2 \big) dx\\
\leq& C_{17}(1+\kappa^{-1})+C_{18}\kappa^{-\frac{4\gamma-1}{3}-\frac{3}{4(\gamma-1)}}\left(1+\kappa^{-\frac{(2\gamma-1)(4\gamma-3)+2}{3}}\right).
\end{align*}
From the standard ODE argument this entails that for $t>0$
 \begin{align*}
&\|\Delta \m(\cdot, t)\|_{L^2}^2+\|\nabla \m(\cdot, t)\|_{L^2}^2\\
\le&\max\left\{\|\Delta\m_0\|_{L^2}^2+\|\nabla\m_0\|_{L^2}^2,\frac{C_{17}(1+\kappa^{-1})+ C_{18}\kappa^{-\frac{4\gamma-1}{3}-\frac{3}{4(\gamma-1)}}\left(1+\kappa^{-\frac{(2\gamma-1)(4\gamma-3)+2}{3}}\right)}{C_{19}\kappa}\right\}.
\end{align*}
Note that due to $\gamma>1$, it holds that
$$\frac{4\gamma-1}{3}+\frac{3}{4(\gamma-1)}+\frac{(2\gamma-1)(4\gamma-3)+2}{3}=\frac{8\gamma^2-6\gamma+4}{3}+\frac{3}{4(\gamma-1)}>1.$$
Based on this and Young's inequality, we have \eqref{eq-H2ing} as desired.
\end{proof}

Now we are ready to gain the uniform-in-time boundedness stated by \eqref{eq-bedd} whenever $\kappa$ is appropriately large.
\begin{lemma}\label{le-ubd}
 Let  $\hat{\kappa}_0$ be the positive constant determined by Lemma \ref{le-H2P}, then whenever $\kappa\ge\hat{\kappa}_0$,  there exists $C>0$  independent of $t$ such that
  \begin{equation}\label{eq-bedd-}
 \|\m(\cdot,t)\|_{H^2}+\|(\m\cdot\nabla p)(\cdot,t)\|_{H^2}+\|p(\cdot,t)\|_{H^3}\leq C,\quad t>0.
 \end{equation}
\end{lemma}

\begin{proof}
It follows from  \eqref{eq-ph2}, \eqref{eq-ph3} and Young's inequality  that there exists $C_1>0$ independent of $\kappa$ and $t$ such that
\begin{align*}
&\|\Delta p(\cdot,t)\|_{L^2}+ \|\nabla\Delta p(\cdot,t)\|_{L^2}+\|\Delta(\m\cdot\nabla p)(\cdot,t)\|_{L^2}\\
\leq& C_1\left(1+\left(1+\kappa^{-1}\right)\left(1+\|\Delta\m(\cdot,t)\|_{L^2}^2\right)\right),\quad t>0,
\end{align*}
which, combined with Lemma  \ref{le-H2P} and the elliptic estimates \eqref{eq-esp} and \eqref{eq-esm}, ensures that  whenever $\kappa>\hat{\kappa}_0$, there exists $C_2>0$ such that
\begin{align*}
    \|p(\cdot,t)\|_{H^3}+\|(\m\cdot\nabla p)(\cdot,t)\|_{H^2}\le C_2(1+\kappa^{-2}),\,\,\quad t>0,\,\,\textrm{if}\,\,\gamma=1,
\end{align*}
and
\begin{align*}
    \|p(\cdot,t)\|_{H^3}++\|(\m\cdot\nabla p)(\cdot,t)\|_{H^2}\le C_2(1+\kappa^{-2-\frac{8\gamma^2-6\gamma+4}{3}-\frac{6}{8(\gamma-1)}}),\,\,\quad t>0,\,\,\textrm{if}\,\,\gamma>1.
\end{align*}
Similarly, due to the elliptic estimate \eqref{eq-elliptic2}, whenever $\kappa>\hat{\kappa}_0$ we have
\begin{align*}
\|\m(\cdot,t)\|_{H^2} \leq C_3(1+\kappa^{-\frac12}),\quad t>0, \,\,\textrm{if}\,\,\gamma=1,
\end{align*}
and
\begin{align*}
\|\m(\cdot,t)\|_{H^2} \leq C_3\left(1+\kappa^{-\frac12-\frac{4\gamma^2-3\gamma+2}{3}-\frac{3}{8(\gamma-1)}}\right),\quad t>0,\,\,\textrm{if}\,\,\gamma>1,
\end{align*}
with $C_3>0$ independent of $\kappa$ and $t$. Consequently, we have \eqref{eq-bedd-} as desired.
\end{proof}

At the end of this section, we prove the uniqueness of classical solutions, i.e., the second statement of Theorem \ref{th-global}. The uniform bound of the solution component $\m$ provided by Lemma \ref{le-H2P} will also play a significant role in the proof.

\begin{lemma}\label{le-umL2c}
Let $\hat{\kappa}_0$ be determined by Lemma \ref{le-H2P}. Then there exists $\hat{\kappa}_1\ge\hat{\kappa}_0$, such that if $\kappa\geq\hat{\kappa}_1$, then the  classical solution  of the initial-boundary value problem \eqref{eq-ib0} is unique.
\end{lemma}

\begin{proof}
Let $(\m_1, p_1)$ and $(\m_2, p_2)$ be classical
solutions to the initial-boundary value problem \eqref{eq-ib0}, with the same initial data. For convenience, we denote
\[\delta\m=\m_1-\m_2 \quad\mathrm{and}\quad\delta p=p_1-p_2.\]
Obviously, according to \eqref{eq-ib0} and using $(\m_i\otimes\m_i)\nabla p_i=(\m_i\cdot\nabla p_i)\m_i(i=1, 2)$, 
 $(\delta\m,\delta p)$ satisfies the following initial-boundary value problem
\begin{equation}\label{eq-asy}
 \left\{
\begin{split}
&(\delta\m)_t-\kappa\Delta\delta\m+\big(|\m_1|^{2(\gamma-1)}\m_1-|\m_2|^{2(\gamma-1)}\m_2\big)\\
&\qquad=(\m_1\cdot\nabla p_1)\nabla p_1-(\m_2\cdot\nabla p_2)\nabla p_2,&x\in\Omega,\,\,t>0,\\
&-\Delta \delta p=\nabla\cdot\left[(\m_1\cdot\nabla p_1)\m_1-(\m_2\cdot\nabla p_2)\m_2\right],&x\in\Omega,\,\,t>0,\\
&\delta\m(x,0)={\bf0},&x\in\Omega,\\
&\delta\m={\bf0},\quad \delta p=0,&x\in\partial\Omega,\,\,t>0.
\end{split}
\right.
\end{equation}
Taking $L^2$ inner product of the first equation in \eqref{eq-asy}  with $\delta \m$, we obtain
\begin{align}\label{eq-asy-1}
&\frac12\frac{d}{dt}\int_\Omega|\delta\m|^2dx+\kappa\int_\Omega|\nabla\delta\m|^2dx+\int_\Omega\big(|\m_1|^{2(\gamma-1)}\m_1-|\m_2|^{2(\gamma-1)}\m_2\big)\cdot\delta\m dx\nonumber\\
=&\int_\Omega\big((\m_1\cdot\nabla p_1)\nabla p_1-(\m_2\cdot\nabla p_2)\nabla p_2\big)\cdot\delta \m dx.
\end{align}
Here, we recall an elementary inequality using convexity and superlinearity, see for instance \cite{Oden1986}, that for any $x,y\in\R^2$
 \begin{align}\label{eq-ele1}
\left(|x|^{2(\gamma-1)}x-|y|^{2(\gamma-1)}y\right)\cdot(x-y)\geq C|x-y|^{2\gamma},\quad \gamma\geq1,
\end{align}
for some $C$ depending only on $\gamma$.
By virtue of this, we have
\begin{align}\label{eq-asy-2}
\int_\Omega\big(|\m_1|^{2(\gamma-1)}\m_1-|\m_2|^{2(\gamma-1)}\m_2\big)\cdot\delta\m dx\geq C\|\delta \m\|_{L^{2\gamma}}^{2\gamma}\ge0.
\end{align}
Moreover, an application of  H\"{o}lder's inequality and Young's inequality leads to
\begin{align*}
&\int_\Omega\big((\m_1\cdot\nabla p_1)\nabla p_1-(\m_2\cdot\nabla p_2)\nabla p_2\big)\cdot\delta \m dx\\
=&\int_\Omega\big((\delta\m\cdot\nabla p_1)\nabla p_1+(\m_2\cdot\nabla\delta p)\nabla p_1+(\m_2\cdot\nabla p_2)\nabla\delta p\big)\cdot\delta \m dx\\
\leq&\|\delta\m\|_{L^4}^2\|\nabla p_1\|_{L^4}^2+\|\m_2\|_{L^\infty}\|\nabla\delta p\|_{L^2}\|\nabla p_1\|_{L^4}\|\delta \m\|_{L^4}+\|\m_2\|_{L^\infty}\|\nabla p_2\|_{L^4}\|\nabla\delta p\|_{L^2}\|\delta \m\|_{L^4}\\
\leq&\|\delta\m\|_{L^4}^2\big(\|\nabla p_1\|_{L^4}^2+\|\m_2\|_{L^\infty}^2\|\nabla p_1\|_{L^4}^2+\|\m_2\|_{L^\infty}^2\|\nabla p_2\|_{L^4}^2\big)+\frac12\|\nabla\delta p\|_{L^2}^2.
\end{align*}
Next, thanks to  $\m_i|_{\partial\Omega}={\bf0}$ and $p_i|_{\partial\Omega}=0$ for $i\in\{1, 2\}$,  the Poincar\'{e}'s inequality implies
\[\|\delta\m\|_{L^4}\le C_1\|\nabla \delta\m\|_{L^2},\]
the Sobolev embedding $H^2\hookrightarrow \hookrightarrow L^\infty$ and the elliptic estimate \eqref{eq-elliptic2} entail that 
\[\|\m_2\|_{L^\infty}\le C_2\|\m_2\|_{H^2}\le C_3\|\Delta\m_2\|_{L^2},\]
and the Sobolev embedding $H^1\hookrightarrow \hookrightarrow L^4$ and the elliptic estimates \eqref{eq-elliptic1} assure that
\[\|\nabla p_i\|_{L^4}\le C_4\|\nabla p_i\|_{H^1}\le C_5\|\Delta p_i\|_{L^2},\]
where $C_j>0$ for $j\in\{1, 2, 3, 4, 5\}$ depending only on $\Omega$,  these together lead to
\begin{align}\label{eq-asy-3}
&\int_\Omega\big((\m_1\cdot\nabla p_1)\nabla p_1-(\m_2\cdot\nabla p_2)\nabla p_2\big)\cdot\delta \m dx\nonumber\\
\le& C_6\|\nabla \delta\m\|_{L^2}^2\big(\|\Delta p_1\|_{L^2}^2+\|\Delta\m_2\|_{L^2}^2\|\Delta p_1\|_{L^2}^2+\|\Delta\m_2\|_{L^2}^2\|\Delta p_2\|_{L^2}^2\big)+\frac12\|\nabla\delta p\|_{L^2}^2
\end{align}
with some $C_6>0$ depending only on $\Omega$. Combining \eqref{eq-asy-1}, \eqref{eq-asy-2} and \eqref{eq-asy-3}, we arrive at
\begin{align}\label{eq-umL2}
&\frac{d}{dt}\int_\Omega|\delta\m|^2dx+2\kappa\int_\Omega|\nabla\delta\m|^2dx\nonumber\\
\leq&2C_6\|\nabla\delta \m\|_{L^2}^2\big(\|\Delta p_1\|_{L^2}^2+\|\Delta\m_2\|_{L^2}^2\|\Delta p_1\|_{L^2}^2+\|\Delta\m_2\|_{L^2}^2\|\Delta p_2\|_{L^2}^2\big)+\|\nabla\delta p\|_{L^2}^2\nonumber\\
\le&2C_6\|\nabla\delta \m\|_{L^2}^2\big(\|\Delta p_1\|_{L^2}^2+\|\Delta p_2\|_{L^2}^2\big)\big(1+\|\Delta\m_2\|_{L^2}^2\big)+\|\nabla\delta p\|_{L^2}^2.
\end{align}

Similarly, to estimate $\|\nabla\delta p\|_{L^2}$, we test the second equation in \eqref{eq-asy} by $\delta p$, and use the  identity 
\[(\m_1\cdot\nabla p_1)\m_1-(\m_2\cdot\nabla p_2)\m_2=(\m_1\cdot\nabla\delta p)\m_1+(\delta\m\cdot\nabla p_2)\m_1+(\m_2\cdot\nabla p_2)\delta\m\]
and H\"{o}lder's inequality to get
\begin{align*}
\|\nabla\delta p\|_{L^2}^2+\|\m_1\cdot\nabla\delta p\|_{L^2}^2
=&-\int_\Omega(\delta\m\cdot\nabla p_2)\m_1\cdot\nabla\delta p dx-\int_\Omega(\m_2\cdot\nabla p_2)\delta \m\cdot\nabla\delta pdx\\
\le&\|\delta\m\|_{L^4}\|\nabla p_2\|_{L^4}\|\m_1\|_{L^\infty}\|\nabla\delta p\|_{L^2} +\|\m_2\|_{L^\infty}\|\nabla p_2\|_{L^4}\|\delta\m\|_{L^4}\|\nabla\delta p\|_{L^2},
\end{align*}
which, combined with Young's inequality, Sobolev's embedding, Poincar\'{e}'s inequality and elliptic estimates again, guarantees there exists $C_7>0$ dependening only on $\Omega$ such that
\begin{align}\label{eq-umL2-0}
\|\nabla\delta p\|_{L^2}^2\leq& C_7 \|\nabla\delta \m\|_{L^2}^2\|\Delta p_2\|_{L^2}^2\big(\|\Delta\m_1\|_{L^2}^2+\|\Delta\m_2\|_{L^2}^2\big).
\end{align}
Inserting it into \eqref{eq-umL2} and using \eqref{eq-ph2}, we attain that
\begin{align*}
\frac{d}{dt}\int_\Omega|\delta\m|^2dx +& 2\kappa\int_\Omega|\nabla\delta\m|^2dx\\
&\leq C_8 \big(\|\Delta p_1\|_{L^2}^2+\|\Delta p_2\|_{L^2}^2\big)\big(1+\|\Delta\m_1\|_{L^2}^2+\|\Delta\m_2\|_{L^2}^2\big)\|\nabla\delta \m\|_{L^2}^2\\
&\leq C_9 \big(1+(1+\kappa^{-1})(\|\Delta\m_1\|_{L^2}^2+\|\Delta\m_2\|_{L^2}^2)\big)\big(1+\|\Delta\m_1\|_{L^2}^2+\|\Delta\m_2\|_{L^2}^2\big)\|\nabla\delta \m\|_{L^2}^2\\
&\leq  C_{10} \big(1+(1+\kappa^{-2})(\|\Delta\m_1\|_{L^2}^4+\|\Delta\m_2\|_{L^2}^4)\big)\|\nabla\delta \m\|_{L^2}^2
\end{align*}
with some $C_8,\,C_9,\,C_{10}>0$ independent of $\kappa$ and $t$, which, combined with \eqref{eq-H2in} and \eqref{eq-H2ing}, ensures that for $\kappa\ge\hat{\kappa}_0$ (determined by Lemma \ref{le-H2P}) there exists $C_{11}>0$ independent of $t$ and $\kappa$ such that
\begin{align*}
\frac{d}{dt}\int_\Omega|\delta\m|^2dx+2\kappa\int_\Omega|\nabla\delta\m|^2dx\leq& C_{11} (1+\kappa^{-2}g^4(\kappa,\gamma))\|\nabla\delta \m\|_{L^2}^2,
\end{align*}
 where $g(\kappa,\gamma)=\kappa^{-\frac12}$ if $\gamma=1$ and $g(\kappa,\gamma)=\kappa^{-\frac12-\frac{4\gamma^2-3\gamma+2}{3}-\frac{3}{8(\gamma-1)}}$  if $\gamma>1$. Hence, there must exist $\hat{\kappa}_1\geq\hat{\kappa}_0$ such that if $\kappa\geq\hat{\kappa}_1$, then $C_{11}(1+\kappa^{-2})g^4(\kappa,\gamma)\leq\kappa$, and thereby
\begin{align*}
\frac{d}{dt}\int_\Omega|\delta\m|^2dx+\kappa\int_\Omega|\nabla\delta\m|^2dx\leq0.
\end{align*}
It entails that
\[
\int_\Omega|\delta\m(\cdot,t)|^2dx+\kappa\int_0^t\int_\Omega|\nabla\delta\m|^2dx\leq\|\delta\m_0\|_{L^2}=0,
\]
 which implies $\int_\Omega|\delta\m(\cdot,t)|^2dx=0$ and $\int_\Omega|\nabla\delta\m|^2dx=0$, the latter combined with \eqref{eq-umL2-0} further ensures  $\|\nabla\delta p\|_{L^2}^2=0$. This fact together with Poincar\'{e}'s inequality allow us to deduce that $\delta\m={\bf0}$ and $\delta p=0$, namely, the classical solution of the initial-boundary value problem \eqref{eq-ib0} is unique.
\end{proof}

\section{Uniqueness of stationary solutions}\label{se-stss}
In this section we plan on exploring the uniqueness  of  stationary solutions when $\kappa$ is suitably large. As a preparation for our analysis,  we recall the existence of  classical stationary solutions, see \cite{Xu2018}.

\begin{lemma}\label{le-sexit}
If $S\in H^2$, then the stationary problem  \eqref{eq-ss} possesses at least one classical solution $(\m_\infty,p_\infty)$ in the sense that $\m_\infty\in\left( \mathcal{C}^{2+\alpha}(\overline{\Omega})\right)^2$ and $p_\infty\in \mathcal{C}^{2+\alpha}(\overline{\Omega})$ for some $\alpha\in(0,1)$.
\end{lemma}

Similar to the proof of uniqueness of classical solutions to the initial-boundary value problem \eqref{eq-ib0}, in obtaining the uniqueness of classical stationary solutions  the dependence on $\kappa$ of the $H^2$-norm of solutions plays a significant role.  
Since classical solutions $(\m_\infty, p_\infty)$  to the stationary problem \eqref{eq-ss} can be considered as classical solutions to the initial-boundary value problem \eqref{eq-ib0} with initial data $(\m_0, p_0)=(\m_\infty, p_\infty)$, therefore estimates for solutions to \eqref{eq-ib0} that are independent of initial data are also valid for stationary solutions; however, estimates depending on initial data cannot be used directly.

\begin{lemma}\label{le-sPL6t}
There exists a positive constant $C$, independent of  $\kappa$, such that
\begin{align}\label{eq-sHL6t}
\|\nabla p_\infty\|_{L^2}+\|\m_\infty\cdot\nabla p_\infty\|_{L^2}\leq C,\\
\label{eq-sHL6t0}
\kappa\|\nabla\m_\infty\|_{L^2}^2+ \|\m_\infty\|_{L^{2\gamma}}^{2\gamma}\leq C.
\end{align}
\end{lemma}

\begin{proof}
Since from the proof of  \eqref{eq-HL6t} one can see the constant $C$ in \eqref{eq-HL6t} is indeed independent of the initial data and $\kappa$, thus we have \eqref{eq-sHL6t} as desired. However, we cannot directly obtain \eqref{eq-sHL6t0} from \eqref{eq-HL6t0-1} because  the constant $C$ therein actually depends on initial data. Therefore, let us give a brief proof of it.

Testing the first equation in \eqref{eq-ss} by $\m_\infty$, we have
\[
 \kappa\int_\Omega|\nabla\m_\infty|^2dx+\int_\Omega|\m_\infty|^{2\gamma}dx= \int_\Omega(\m_\infty\cdot\nabla p_\infty)^2dx,
\]
which, using \eqref{eq-sHL6t}, entails \eqref{eq-sHL6t0} immediately.
\end{proof}

On the basis of Lemma \ref{le-sPL6t}, we have the following estimates.

 \begin{lemma}\label{le-sPH2}
 There exists $\widehat{\kappa}>0$, with the property that, whenever $\kappa\geq\widehat{\kappa}$,
one can find $C>0$, independent of  $\kappa$, such that
\begin{align}\label{eq-sHH2}
\kappa\|\Delta\m_\infty\|_{L^2}^2+\|\Delta p_\infty\|_{L^2}^2+\|\nabla(\m_\infty\cdot\nabla p_\infty)\|_{L^2}^2\leq C(1+\kappa^{-1}).
\end{align}
\end{lemma}

\begin{proof}
According to \eqref{eq-ph2-t}, we first have
\begin{align}\label{eq-sph2}
\int_\Omega|\Delta p_\infty|^2dx\,+ &\int_\Omega|\nabla(\m_\infty\cdot\nabla p_\infty)|^2dx\\
=&
\sum_{i=1}^2\int_\Omega\big[\partial_i(\m_\infty\cdot\nabla p_\infty)\partial_i\m_\infty\cdot\nabla p_\infty-(\m_\infty\cdot\nabla p_\infty)\partial_i\m_\infty\cdot\nabla\partial_i p_\infty dx-\int_\Omega S\Delta p_\infty dx.\nonumber
\end{align}
Next, taking  the $L^2$ inner product of the first equation in \eqref{eq-ss}   with
$-\Delta\m_\infty$ and  integrating by parts, we obtain
\begin{align*}
\kappa\|\Delta\m_\infty\|_{L^2}^2\,+&\int_\Omega|\m_\infty|^{2(\gamma-1)}\m_\infty\cdot(-\Delta\m_\infty)dx\\
=&\sum_{i=1}^2\int_\Omega\partial_i(\m_\infty\cdot\nabla p_\infty)\nabla p_\infty\cdot\partial_i\m_\infty+(\m_\infty\cdot\nabla p_\infty)\nabla\partial_i p_\infty\cdot\partial_i\m_\infty dx,
\end{align*}
which, combined  with \eqref{eq-sph2}, ensures
\begin{align*}
\kappa\|\Delta\m_\infty\|_{L^2}^2\,+&\int_\Omega|\m_\infty|^{2(\gamma-1)}\m_\infty\cdot(-\Delta\m_\infty)dx+\int_\Omega|\Delta p_\infty|^2dx+\int_\Omega|\nabla(\m_\infty\cdot\nabla p_\infty)|^2dx\\
=&2\sum_{i=1}^2\int_\Omega\partial_i(\m_\infty\cdot\nabla p_\infty)\nabla p_\infty\cdot\partial_i\m_\infty dx-\int_\Omega S\Delta p_\infty dx.
\end{align*}
A direct calculation shows that
\begin{align*}
\int_\Omega |\m_\infty|^{2(\gamma-1)}&\m_\infty\cdot(-\Delta\m_\infty)dx\\
 =&\int_\Omega |\m_\infty|^{2(\gamma-1)}|\nabla\m_\infty|^2 dx+2(\gamma-1)\sum_{i=1}^2\int_\Omega |\m_\infty|^{2(\gamma-2)}(\m_\infty\cdot\partial_i\m_\infty)^2dx\geq0,
\end{align*}
due to $\gamma\geq1$.  Hence, it follows from H\"{o}lder's inequality that
\begin{align*}
\kappa\|\Delta\m_\infty\|_{L^2}^2\,&+\int_\Omega|\Delta p_\infty|^2dx+\int_\Omega|\nabla(\m_\infty\cdot\nabla p_\infty)|^2dx\\
&\leq 2\sum_{i=1}^2\|\partial_i(\m_\infty\cdot\nabla p_\infty)\|_{L^2}\|\nabla p_\infty\|_{L^4}\|\partial_i\m_\infty\|_{L^4}+\|S\|_{L^2}\|\Delta p_\infty\|_{L^2}\\
&\le 2\|\nabla(\m_\infty\cdot\nabla p_\infty)\|_{L^2}\|\nabla p_\infty\|_{L^4}\|\nabla\m_\infty\|_{L^4}+\|S\|_{L^2}\|\Delta p_\infty\|_{L^2}
\end{align*}
which, together with Young's inequality and  the Gagliardo-Nirenberg inequality \eqref{GNineq-1}, entails
\begin{align*}
I:=\kappa\|\Delta\m_\infty\|_{L^2}^2+\frac12\int_\Omega|\Delta p_\infty|^2dx+&\frac12\int_\Omega|\nabla(\m_\infty\cdot\nabla p_\infty)|^2dx\\
&\leq 2\|\nabla p_\infty\|_{L^4}^2\|\nabla\m_\infty\|_{L^4}^2+\frac12\|S\|_{L^2}^2\\
&\leq C_1\|\nabla p_\infty\|_{L^2}\|\nabla p_\infty\|_{H^1}\|\nabla\m_\infty\|_{L^2}\|\nabla\m_\infty\|_{H^1}+C_2
\end{align*}
with some $C_1,\,C_2>0$ independent of $\kappa$. A combination  with \eqref{eq-sHL6t} and \eqref{eq-sHL6t0} yields  $C_3>0$ independent of $\kappa$ such that
\begin{align*}
I\leq C_3(1+\kappa^{-\frac12})\left(\|\nabla p_\infty\|_{H^1}\|\nabla\m_\infty\|_{H^1}+1\right).
\end{align*}
Based on this, employing the elliptic estimates \eqref{eq-elliptic1} and \eqref{eq-elliptic2} as well as Young's inequality again, we can find $C_4,\,C_5>0$, independent of $\kappa$, such that
\begin{align*}
I\leq C_4(1+\kappa^{-\frac12})\left(\|\Delta p_\infty\|_{L^2}\|\Delta\m_\infty\|_{L^2}+1\right)
\leq \frac14\|\Delta p_\infty\|_{L^2}^2+C_5(1+\kappa^{-1})\left(\|\Delta\m_\infty\|_{L^2}^2+1\right).
\end{align*}
Now we can pick $\kappa_4>0$ such that whenever $\kappa\geq \kappa_4$, then $C_5(1+\kappa^{-1})\leq \frac\kappa2$, and thereby
\begin{align*}
I=\frac\kappa2\|\Delta\m_\infty\|_{L^2}^2+\frac14\int_\Omega|\Delta p_\infty|^2dx+\frac12\int_\Omega|\nabla(\m_\infty\cdot\nabla p_\infty)|^2dx
\leq C_5(1+\kappa^{-1}),
\end{align*}
and thus finishes the proof of Lemma \ref{le-sPH2}.
\end{proof}

In light of Lemma \ref{le-sPL6t} and Lemma \ref{le-sPH2},  it is easy to show the uniqueness of the stationary solution.

\begin{lemma}\label{le-gPH1}
Let $\widehat{\kappa}$ be determined by Lemma \ref{le-sPH2}. Then there exists $\widehat{\kappa}_1\geq\widehat{\kappa}$ with the property that, whenever $\kappa\geq\widehat{\kappa}_1$, the solution of the stationary problem
 \eqref{eq-ss}
 is unique.
\end{lemma}

\begin{proof}
 Assume that $(\m_{1\infty},p_{1\infty})$ and  $(\m_{2\infty},p_{2\infty})$ are classical solutions of the stationary problem \eqref{eq-ss}.
Setting $\delta\m_\infty:=\m_{1\infty}-\m_{2\infty}$ and $\delta p_\infty:=p_{1\infty}-p_{2\infty}$, it is clear that $(\delta\m_\infty,\delta p_\infty)$ satisfies the following elliptic problem
\begin{equation}\label{eq-delss}
\left\{
\begin{split}
&- \kappa\Delta\delta\m_{1\infty}+|\m_{1\infty}|^{2(\gamma-1)}\m_{1\infty}-|\m_{2\infty}|^{2(\gamma-1)}\m_{2\infty}\\
&\quad \quad\quad \quad\quad = (\m_{1\infty}\cdot\nabla p_{1\infty})\nabla p_{1\infty}-(\m_{2\infty}\cdot\nabla p_{2\infty})\nabla p_{2\infty},& x\in \Omega,\\
&-\Delta\delta p_\infty-\nabla\cdot\left[(\m_{1\infty}\cdot\nabla p_{1\infty})\m_{1\infty}-(\m_{2\infty}\cdot\nabla p_{2\infty})\m_{2\infty}\right]=0,& x\in\Omega,\\
&\delta\m_\infty={\bf0},\,\,\,\delta p_\infty=0,& x\in \partial\Omega.
\end{split}
\right.
\end{equation}
Along the proof of Lemma \ref{le-umL2c}, we can easily obtain from \eqref{eq-umL2} that
\begin{align}\label{eq-degmi}
\kappa\|\nabla\delta\m_{\infty}\|_{L^2}^2
\le C_1\|\nabla\delta \m_{\infty}\|_{L^2}^2\big(\|\Delta p_{1 \infty}\|_{L^2}^2+\|\Delta p_{2\infty}\|_{L^2}^2\big)\big(1+\|\Delta\m_{2\infty}\|_{L^2}^2\big)+\frac12\|\nabla\delta p_{\infty}\|_{L^2}^2
\end{align}
for some $C_1>0$ depending only on $\Omega$. In addition, according to \eqref{eq-umL2-0} we have
\begin{align*}
\|\nabla\delta p_\infty\|_{L^2}^2\leq& C_2 \|\nabla\delta \m_\infty\|_{L^2}^2\|\Delta p_{2\infty}\|_{L^2}^2\big(\|\Delta\m_{1\infty}\|_{L^2}^2+\|\Delta\m_{2\infty}\|_{L^2}^2\big)
\end{align*}
for some $C_2>0$ also depending only on $\Omega$. Combining this and \eqref{eq-degmi}, we arrive at
\begin{align*}
\kappa\|\nabla\delta\m_{\infty}\|_{L^2}^2\,&+\frac12\|\nabla\delta p_\infty\|_{L^2}^2\\
&\leq C_3\|\nabla\delta \m_{\infty}\|_{L^2}^2\big(\|\Delta p_{1 \infty}\|_{L^2}^2+\|\Delta p_{2\infty}\|_{L^2}^2\big)\big(1+\|\Delta\m_{1\infty}\|_{L^2}^2+\|\Delta\m_{2\infty}\|_{L^2}^2\big)
\end{align*}
for $C_3=\max\{C_1, C_2\}$. Thanks to \eqref{eq-sHH2}, whenever $\kappa>\widehat{\kappa}$, there exists $C_4>0$ independent of $\kappa$ such that 
\[
\|\Delta\m_{i\infty}\|_{L^2}^2\le C_4\kappa^{-1}(1+\kappa^{-1})
\]
and
\[
\|\Delta p_{i\infty}\|_{L^2}^2\le C_4(1+\kappa^{-1})
\]
for $i\in\{1, 2\}$. These together lead to
    \begin{align*}
\kappa\|\nabla\delta\m_{\infty}\|_{L^2}^2+\frac12\|\nabla\delta p_\infty\|_{L^2}^2
\leq& C_5\|\nabla\delta \m_{\infty}\|_{L^2}^2(1+\kappa^{-1})\big(1+\kappa^{-1}(1+\kappa^{-1})\big)
\end{align*}
for some $C_5>0$ independent of $\kappa$. There definitely exists $\widehat{\kappa}_1\geq\widehat{\kappa}$ such that 
\[
C_5(1+\kappa^{-1})\big(1+\kappa^{-1}(1+\kappa^{-1})\big)\le\frac{\kappa}2
\]
whenever $\kappa\ge\widehat{\kappa}_1$,  we thus have
 \[
\kappa\|\nabla\delta\m_{\infty}\|_{L^2}^2+\|\nabla\delta p_\infty\|_{L^2}^2\le 0.
\]
This, together with Poincar\'{e}'s inequality, immediately implies
\begin{align*}
\kappa\|\delta\m_{\infty}\|_{L^2}^2+\|\delta p_\infty\|_{L^2}^2
\leq&0,
\end{align*}
which ensures the uniqueness, and thus finishes the proof of Lemma \ref{le-gPH1}.
\end{proof}

Based on  Lemma \ref{le-gPH1}, we can reveal the unique stationary solution is semi-trivial.

\begin{corollary}\label{le-gPHs}
Let $\widehat{\kappa}_1$ be determined by Lemma \ref{le-gPH1}. If $\kappa\geq\widehat{\kappa}_1$, then the unique classical solution of the stationary problem  \eqref{eq-ss} is the semi-trivial steady state $({\bf0},p_\infty^\ast)$, where $p_\infty^\ast$ is the unique solution of the elliptic problem \eqref{eq-ss0}.
\end{corollary}

\begin{proof}
It is clear that $({\bf0},p_\infty^\ast)$ is a  solution of the stationary problem \eqref{eq-ss}. Using the uniqueness of solution of the latter, we conclude as desired.
\end{proof}

\section{Stability of steady state. Proof of Theorem \ref{th-global}}\label{se-lim}

The aim of this final section is to show the global asymptotic stability of the semi-trivial steady state $({\bf0},p_\infty^\ast)$,  where $p_\infty^\ast$ is the unique classical solution of the elliptic problem \eqref{eq-ss0}.

\begin{lemma}\label{le-umL2s}
There exists $\overline{\kappa}>0$ such that if $\kappa\geq\overline{\kappa}$, then the semi-trivial steady state $({\bf0},p_\infty^\ast)$ is globally asymptotic stable, where $p_\infty^\ast$ is the unique classical solution of the elliptic problem \eqref{eq-ss0}. Moreover, there exist $C>0$ (independent of $t$) and $\mu>0$ (independent of $\kappa$ and $t$) such that
\begin{align}\label{eq-demlin}
\|\m(\cdot,t)\|_{L^\infty}+\|p-p_\infty^\ast\|_{H^2}\leq Ce^{-\mu\kappa t},\quad t>0.
\end{align}
\end{lemma}

\begin{proof}
We test the first equation in \eqref{eq-ib0} by $\m$ to get
\[
\frac12\frac{d}{dt}\int_\Omega|\m|^2dx+\kappa\int_\Omega|\nabla\m|^2dx+\int_\Omega|\m|^{2\gamma}dx=\int_\Omega(\m\cdot\nabla p)^2dx.
\]
Using H\"{o}lder's inequality, Sobolev's inequality, Poincar\'{e}'s inequality and the elliptic estimate \eqref{eq-elliptic1}, we obtain $C_1,\,C_2>0$, independent of $\kappa$ and $t$, such that
\[
\int_\Omega(\m\cdot\nabla p)^2dx\leq\|\m\|_{L^4}^2\|\nabla p\|_{L^4}^2\leq C_1\|\nabla \m\|_{L^2}^2\|\nabla p\|_{H^1}^2\leq C_2\|\nabla\m\|_{L^2}^2\|\Delta p\|_{L^2}^2,
\]
which, combined with \eqref{eq-ph2}, \eqref{eq-H2in} and \eqref{eq-H2ing}, leads to
\[
\int_\Omega(\m\cdot\nabla p)^2dx\leq C_3\|\nabla\m\|_{L^2}^2(1+(1+\kappa^{-1})\|\Delta\m\|_{L^2}^2)\leq C_4(1+\kappa^{-1}g^2(\kappa,\gamma))\|\nabla\m\|_{L^2}^2,
\]
with some positive constants $C_3$ and $C_4$ independent of $\kappa$ and $t$, where $g(\kappa,\gamma)=\kappa^{-\frac12}$ if $\gamma=1$ and $g(\kappa,\gamma)=\kappa^{-\frac12-\frac{4\gamma^2-3\gamma+2}{3}-\frac{3}{8(\gamma-1)}}$  if $\gamma>1$.
Now, we arrive at
\[
\frac{d}{dt}\int_\Omega|\m|^2dx+2\kappa\int_\Omega|\nabla\m|^2dx\leq 2C_4(1+\kappa^{-1}g^2(\kappa,\gamma))\|\nabla\m\|_{L^2}^2.
\]
For such $C_4$ and  $\gamma\geq1$, we can find $\kappa_1\geq\widehat{\kappa}_1$ with $\widehat{\kappa}_1$ determined by  Lemma \ref{le-umL2c} such that if $\kappa\geq\kappa_1$, then $2C_4(1+\kappa^{-1}g^2(\kappa,\gamma))\leq\kappa$, and thereby obtain
\[
\frac{d}{dt}\int_\Omega|\m|^2dx+\kappa\int_\Omega|\nabla\m|^2dx\leq 0.
\]
Poincar\'{e}'s inequality enables us to get $C_5>0$, independent of $\kappa$ and $t$,  such that
\[
\frac{d}{dt}\int_\Omega|\m|^2dx+C_5\kappa\int_\Omega|\m|^2dx\leq 0,
\]
which implies
\begin{align}\label{eq-deml2-1}
\int_\Omega|\m(\cdot,t)|^2dx\leq\int_\Omega|\m_0|^2dx e^{-C_5\kappa t},\quad t>0.
\end{align}
Recalling  the Gagliardo-Nirenberg inequality \eqref{GNineq-2}
that
\[
\|\m\|_{L^\infty}\leq C_6\|\m\|_{L^2}^\frac12\|\m\|_{H^2}^\frac12
\]
with some $C_6>0$ independent of $\kappa$ and $t$, we infer from the uniform boundedness of $\|\m\|_{H^2}$ provided by Lemma \ref{le-H2P} and the estimate \eqref{eq-deml2-1} that
\begin{align}\label{eq-deml2}
\|\m(\cdot,t)\|_{L^\infty}\leq C_7\|\m\|_{L^2}^\frac12\leq C_8e^{-\frac{C_5}2\kappa t},\quad t>0
\end{align}
with $C_7,\,C_8>0$ independent of $t$.
On the other hand, thanks to \eqref{eq-ib0} and \eqref{eq-ss0}  we have
\[
-\nabla\cdot\left[(\mathbf{I}+\m\otimes\m)\nabla p\right]=-\Delta p_\infty^\ast.
\]
Testing it by $p-p_\infty^\ast$, we obtain
\[
\int_\Omega|\nabla(p-p_\infty^\ast)|^2dx=\int_\Omega\left[(\m\otimes\m)\nabla p\right]\cdot \nabla(p-p_\infty^\ast)dx,
\]
which, together with H\"{o}lder's inequality, entails
\[
\|\nabla(p-p_\infty^\ast)\|_{L^2}\leq\|(\m\otimes\m)\nabla p\|_{L^2}\leq\|\m\|_{L^\infty}^2\|\nabla p\|_{L^2}.
\]
This, together with the uniform boundedness of $\|\nabla p\|_{L^2}$ provided by \eqref{eq-HL6t} and  the estimate \eqref{eq-deml2}, implies that
\begin{align}\label{eq-deml2-2}
\|\nabla(p-p_\infty^\ast)(\cdot,t)\|_{L^2}\leq C_9 e^{-C_5\kappa t},\quad t>0
\end{align}
with some $C_9>0$ independent of $t$.  
An application of the Poincare inequality and the Gagliardo-Nirenberg inequality \eqref{GNineq-6} yields a positive constant $C_{10}$, independent of $\kappa$ and $t$, such that
\[
\|\nabla(p-p_\infty^\ast)\|_{H^1}\leq C_{10}\|\nabla(p-p_\infty^\ast)\|_{L^2}^\frac12\|\nabla(p-p_\infty^\ast)\|_{H^2}^\frac12,
\]
which, combined the elliptic estimate \eqref{eq-esp}, gives us
\[
\|\nabla(p-p_\infty^\ast)\|_{H^1}^2\leq C_{11}\|\nabla(p-p_\infty^\ast)\|_{L^2}\left(\|D^3(p-p_\infty^\ast)\|_{L^2}+\|\Delta(p-p_\infty^\ast)\|_{L^2}\right),
\]
with some $C_{11}>0$ independent of $\kappa$ and $t$. Note that for given $S\in H^1$ it is clear that $\|p_\infty^\ast\|_{H^3}\leq C_{12}\|S\|_{H^1}$ with some $C_{12}>0$ independent of $\kappa$ and $t$. Hence, we can infer from Lemmas \ref{le-ph2}, \ref{le-ph3} and \ref{le-H2P} that there exist $C_{13},\,C_{14}>0$, independent of $t$, such that
\[
\|\nabla(p-p_\infty^\ast)\|_{H^1}^2\leq C_{13}\|\nabla(p-p_\infty^\ast)\|_{L^2}\left(1+\|\Delta\m\|_{L^2}^2\right)\leq C_{14}\|\nabla(p-p_\infty^\ast)\|_{L^2},
\]
and thereby, it follows from \eqref{eq-deml2-2} that
\begin{align*}
\|\nabla(p-p_\infty^\ast)(\cdot,t)\|_{H^1}\leq C_{15}e^{-\frac{C_5}2\kappa t},\quad t>0,
\end{align*}
with some $C_{15}>0$ independent of $t$. This, combined with Poincar\'{e}'s inequality and \eqref{eq-deml2}, ensures \eqref{eq-demlin} as desired.
\end{proof}

We now have all the necessary tools to prove Theorem \ref{th-global}.

\begin{proof}[Proof of Theorem \ref{th-global}]
In essence, the  uniform-in-time boundedness \eqref{eq-bedd}, the uniqueness of classical solution, the uniqueness of the steady state, and its  global asymptotic stability \eqref{eq-ltime} follow from  Lemmas \ref{le-ubd}, \ref{le-umL2c}, \ref{le-gPH1} and \ref{le-umL2s}, respectively.
Invoking these, we finish  the proof of Theorem \ref{th-global}.
\end{proof}

\section*{Acknowledgements}
JAC was supported by the Advanced Grant Nonlocal-CPD
(Nonlocal PDEs for Complex Particle Dynamics: Phase Transitions, Patterns and Synchronization)
 of the European Research Council Executive Agency (ERC) under the European Union Horizon 2020 research and innovation programme (grant agreement No. 883363), and  partially supported by the EPSRC (grant numbers EP/T022132/1 and EP/V051121/1)
and by the ``Maria de Maeztu'' Excellence Unit IMAG, reference CEX2020-001105-M, funded
by MCIN/AEI/10.13039/501100011033/.
BL waw supported by  Natural Science Foundation of Ningbo Municipality (No. 2022J147).
 LX was supported by  National Natural Science Foundation of China (No. 11701461), Chongqing Science and Technology Commission Project (No.  CSTB2023NSCQ-MSX0411) and partially supported by  Chongqing Talent Program (No. cstc2024ycjh-bgzxm0046).
 
\appendix 
\section{Appendix}
\renewcommand{\thetable}{A\arabic{table}}

We first remind the reader the classical GNS inequalities frequently used in this work.

\begin{lemma} (Gagliardo-Nirenberg inequality \cite[p.37]{Henry1981})
    Let $\Omega$ be a bounded domain in $\mathbb{R}^n$ with smooth boundary, and let $u\in W^{m, q}\cap L^r$. Then there hold
    \begin{align}\label{GNineq}
        \|u\|_{W^{k, l}}\le C\|u\|_{W^{m, q}}^\theta\|u\|_{L^{r}}^{1-\theta}
    \end{align}
    if $k-n/l=\theta(m-n/q)-n(1-\theta)/r$, $1\le l, q, r\le \infty$, and $k/m\le\theta\le1$. The constant $C$ depends only on $\Omega, m, r, q, k$.
    In particular, in the two-dimensional case, we have
    \begin{align}\label{GNineq-1}
        \|u\|_{L^4}\le C\|u\|_{H^1}^{\frac12}\|u\|_{L^{2}}^{\frac12},\\
        \|u\|_{L^\infty}\le C\|u\|_{H^2}^{\frac12}\|u\|_{L^{2}}^{\frac12},\label{GNineq-2}\\
        \|u\|_{L^\infty}\le C\|u\|_{H^2}^{\frac{2}{q+2}}\|u\|_{L^{q}}^{\frac{q}{q+2}},\label{GNineq-5}\\
        \|\nabla u\|_{L^4}\le C\|u\|_{H^2}^{\frac34}\|u\|_{L^{2}}^{\frac14},\label{GNineq-3}\\
        \|\Delta u\|_{L^2}\le C\|\nabla u\|_{H^2}^{\frac12}\|\nabla u\|_{L^{2}}^{\frac12},\label{GNineq-4}\\
        \|  u\|_{H^1}\le C\| u\|_{H^2}^{\frac12}\|u\|_{L^{2}}^{\frac12}.\label{GNineq-6}
    \end{align}
\end{lemma}

For readers' convenience, let us summarize the elliptic estimates from \cite{Trudinger1983} used in our work.

\begin{lemma}\label{le-ellipticesti}
    Let $\Omega\subset\mathbb{R}^n (n\ge1)$ be a bounded domain with  smooth boundary $\partial\Omega$, $u\in H^{l+2}_0(\Omega)$, then there exists $C=C(n, l, \partial\Omega)$ such that
    \begin{align}\label{eq-elliptic}
        \|u\|_{H^{l+2}}\le C\|\Delta u\|_{H^l}.
    \end{align}
\end{lemma}

\begin{proof}
Consider the following elliptic equation with homogeneous Dirichlet boundary condition
\begin{equation*}
\left\{
\begin{split}
&\Delta u=f,\,\, &x\in\Omega,\\
  & u=0,\,\,& x\in\partial\Omega.
\end{split}
\right.
\end{equation*}
 It follows from \cite[Theorem 8.13, Lemma 9.17]{Trudinger1983} that for any $f\in H^l$ there exists a positive constant $C=C(n, l, \partial\Omega)$ such that
\begin{align*}
        \|u\|_{H^{l+2}}\le C(\|u\|_{L^2}+\|f\|_{H^l})
    \end{align*}
and
\begin{align*}
        \|u\|_{H^{2}}\le C\|\Delta u\|_{L^2}.
    \end{align*}
Since $\Delta u=f$, we have \eqref{eq-elliptic} as desired.
\end{proof}

\end{document}